\documentclass{article}
\usepackage{amsthm, amssymb,comment}
\usepackage{graphicx}
\usepackage{epsfig}
\usepackage{psfrag}

\newcommand{\E}{\mathbb E}
\newcommand{\R}{\mathbb R}
\newcommand{\F}{\mathbb F}
\newcommand{\Prob}{\mathbb P}
\newcommand{\bbN}{\mathbb N}
\newcommand{\bbD}{\mathbb D}
\newcommand{\supp}{\mbox{supp}}
\newcommand{\bbT}{\mathbb T}
\newcommand{\cbbT}{\overline{\mathbb T}}
\newcommand{\sA}{\mathcal A}

\newcommand{\sF}{\mathcal F}

\newcommand{\sL}{\mathcal L}
\newcommand{\sM}{\mathcal M}
\newcommand{\sN}{\mathcal N}
\newcommand{\sP}{\mathcal P}
\newcommand{\sQ}{\mathcal Q}
\newcommand{\sV}{\mathcal V}

\newtheorem{thm}{Theorem}[section]
\newtheorem{lem}{Lemma}[section]
\newtheorem{prop}{Proposition}[section]
\newtheorem{rem}{Remark}[section]
\newtheorem{cor}{Corollary}[section]
\newtheorem{eg}{Example}[section]
\newtheorem{defn}{Definition}[section]
\newtheorem{ass}{Assumption}[section]

\begin{document}
\setcounter{footnote}{1}
\title{Mimicking martingales}

\date{\today}

\author{David Hobson \\
University of Warwick, Coventry CV4 7AL, UK.}

\maketitle

\begin{abstract}
Given the univariate marginals of a real-valued, continuous-time martingale, (respectively, a family of measures parameterised by $t \in [0,T]$
which is increasing in convex order, or a double continuum of call prices) we construct a family of pure-jump martingales which mimic that martingale
(respectively, are consistent with the family of measures, or call prices). As an example, we construct a fake Brownian motion.
Then, under a further `dispersion' assumption, we construct the martingale which (within the family of martingales which are consistent with a given set of measures) has the
smallest expected total variation. We also give a path-wise inequality, which in the mathematical finance context yields a model-independent sub-hedge for
an exotic security with payoff equal to the total variation along a realisation of the price process.
 \end{abstract}

\section{The problem}
\label{sec-intro}

In this article we are concerned with the following problem: given a martingale
$M = (M_t)_{0 \leq t \leq T}$ on a stochastic basis $(\Omega^M, \sF^M, \F^M, \Prob^M)$,
construct a {\em fake version} of $M$ (or equivalently a process which {\em mimics} $M$), i.e. construct
(potentially on a new filtered probability space $(\Omega, \sF, \F, \Prob)$) a stochastic process
$X = (X_t)_{0 \leq t \leq T}$ such that $X$ is a $\F$-martingale and the univariate marginals of
$X$ are the same as those of $M$, but such that the joint marginals are different.

The problem can
be reformulated in two further ways. First, instead of beginning with a martingale $M$, we can begin
with a family of laws $(\mu_t)_{0 \leq t \leq T}$ which are increasing in convex order.  Then the aim
is to construct a filtered probability space $(\Omega, \sF, \F, \Prob)$ and a $\F$-martingale
$X= (X_t)_{0 \leq t \leq T}$ on that space such that $\Prob(X_t \leq x) = \mu_t ((-\infty,x])$ for all $t$ and all $x$.
Then we say that $X$ is consistent with the measures $(\mu_t)_{0 \leq t \leq T}$. Second, but closely related,
instead of beginning with a process or a set of measures we can work in the setting of mathematical
finance and start with a double continuum of European call prices $\{ C(t,k) ; 0 \leq t \leq T, 0 \leq k < \infty \}$
which satisfy no-arbitrage conditions. (We assume that we are working with discounted prices, and then the no-arbitrage conditions
are that: for each $t$, $C(t,k)$ is a decreasing convex function with $C'(t, 0+) \geq 1$ and $\lim_{k \uparrow \infty}C(t,k)=0$;
$C(t, 0)$ is a positive constant, independent of $t$; and
$C(t,k)$ is non-decreasing in $t$, for all $k$.) Then the aim is to find a model which is consistent with the given call
prices, i.e. a filtered probability space and a martingale $X$ on that space such that
$\E[(X_t - k)^+] = C(t,k)$ for all $t$ and all $k$. Then we say $X$ is consistent with the family of call prices $\{C(t,k)\}$.

For a given martingale $(M_t)_{0 \leq t \leq T}$ (or family of laws $(\mu_t)_{0 \leq t \leq T}$
or call prices $(C(t,k))_{0 \leq t \leq T, 0 \leq k < \infty}$) there will in general be many martingales $X$ which mimic
$M$ (are consistent with the family of measures or call prices). (One rare exception is if $\mu_t$ is the uniform measure on
the two point set $\{-t,t\}$ for each $t$.) Then an extension of the mimicking martingale problem is to find an extremal
martingale with the correct marginals. Thus, given a path-dependent functional, for example the total vaiation of the path,
the aim is to find the martingale which minimises (the expected value of) the path-dependent functional within the class of martingales
which have the given marginals. A strong motivation for considering problems of this type is from mathematical finance where there
is a literature on robust (model-independent) pricing and hedging.
Typically in the modelling of financial security prices we do not know the
underlying dynamics, but we can observe put and call prices, and from these we can
infer the market-implied law of the security price at maturity times $t$.
If a double continuum of European vanilla option prices is available, then
we may assume that the family of laws $(\mu_t)_{0 \leq t \leq T}$ is given and
the problem is to construct processes which are consistent with this
family. Further, since the mathematical theory of option pricing
postulates that security prices are given as discounted expectations under
an equivalent martingale measure, a natural additional requirement is that
the price process is a martingale. Then the lowest expected value of the path-dependent functional
(taken over models consistent with call prices) is
identified with the lowest arbitrage-free price for an exotic derivative. 

The literature on these problems can be traced back to Gy\"ongy~\cite{Gyongy:86}, Dupire~\cite{Dupire:97}, Madan and Yor~\cite{MadanYor:02}, the
literature on fake Brownian motions (Hamza and Klebaner~\cite{HamzaKlebaner:07}, Albin~\cite{Albin:08}, Oleszkiewicz~\cite{Oleszkiewicz:08}), {\sc Pcoc}s
(Hirsch et al~\cite{HirschProfetaRoynetteYor:11}) and most recently and most relevantly for this paper,
Henri-Labord\`{e}re et al~\cite{Henri-LabordereTanTouzi:15} and K\"{a}llblad et al~\cite{KallbladTanTouzi:15}. In addition there are strong
connections with the literatures on the Skorokhod embedding problem ({\sc Sep}), see Ob{\l}{\'o}j~\cite{Obloj:04} and Hobson~\cite{Hobson:11} for surveys,
martingale optimal transport ({\sc Mot}, Beiglb\"{o}ck et al~\cite{BeiglbockHenriLaborderePenkner:13}, Beiglb\"ock and Juillet~\cite{BeiglbockJuillet:14}),
and robust hedging of options (Hobson~\cite{Hobson:98, Hobson:11}).

Given an It\^{o} process, Gy\"ongy~\cite{Gyongy:86} showed how to construct a diffusion $X$ with the same univariate marginals. Dupire~\cite{Dupire:97}
took a finance viewpoint and argued that there is a unique martingale diffusion process which is consistent with any sufficiently regular set of discounted
call prices.  Madan and Yor~\cite{MadanYor:02} looked for Markov martingales which are consistent with a given set of univariate
marginals and give several constructions inlcuding one based on the Az\'{e}ma-Yor solution of the {\sc Sep}. In the case of Gaussian laws, this construction
yields a discontinuous fake-Brownian motion. Unaware of this result, Hamza and Klebaner~\cite{HamzaKlebaner:07} introduced the fake Brownian motion problem,
and gave a further example of a discontinuous martingale with $N(0,t)$ marginals. Their paper inspired Albin~\cite{Albin:08} and Oleszkiewicz~\cite{Oleszkiewicz:08}
to produce very elegant constructions of continuous non-Markovian martingales with $N(0,t)$ marginals. More generally, the problem of constructing processes
consistent with a given set of marginals (which are increasing in convex order) was called the {\sc Pcoc} problem by Hirsch et
al~\cite{HirschProfetaRoynetteYor:11}. They give several general methods and examples. However there is much scope for further solutions especially
those with optimality properties.

From the mathematical finance perspective the study of extremal consistent martingales can be traced back to
Hobson~\cite{Hobson:98}, see also the survey article~\cite{Hobson:11}. In this literature a path-dependent functional is interpreted as the payoff
of an exotic option and then the lowest price which is consistent with a given set of call option prices and no-arbitrage is the minimum over
all consistent martingale models of the expected value of the path-dependent functional. One of the
significant achievements of this literature is to understand the relationship between this primal pricing problem and a dual hedging problem. In the dual
problem the aim is to find strategies which satisfy
\[ \mbox{Exotic Option Payoff} \geq \mbox{Payoff from Vanilla Options} + \mbox{Payoff from Self-financing Strategy.} \]
Then we can represent the lower bound on the price of the exotic option in terms of the most expensive sub-hedging vanilla option portfolio (the self-financing
part being costless under any martingale model). The fact that European call options (which have prices which depend on univariate marginals) are
often liquidly traded,
whereas derivatives with payoffs depending on two-or-more maturities (e.g. forward starting straddles or barrier options) are relatively less liquidly traded,
provides a strong motivation for
considering mimicking problems where the aim is to match univariate marginals, but not joint marginals.

One approach to solving the mimicking problem is to exploit links with the {\sc Sep}. Given a stochastic process $Y$ and a target law $\mu$, the {\sc Sep}
is to construct a stopping time $\tau$ such that $Y_\tau \sim \mu$. The classic setting is when $Y$ (which we now write as $W$) is a Brownian motion
null at 0 and $\mu$ is a
centred probability measure. Then, using the correspondence between right-continuous martingales and minimal time-changes of Brownian motion, given $\tau$
such that $W_\tau \sim \mu$ we can define $X = (X_{t})_{0 \leq t \leq T} = (W_{A_t})_{0 \leq t \leq T}$ such that $X$ has law $\mu$ at a fixed time:
simply take $A_t = \frac{t}{T-t} \wedge \tau$. Then we have a martingale with marginal $\mu_T$ at time $T$. Moreover, given solutions of the {\sc Sep}
for non-trivial initial laws (the forward-starting {\sc Sep}) we can use concatenation to construct processes with prescribed univariate marginals at a finite set of times
$0 < t_1 < t_2 < \ldots t_n \leq T$.

Following work of Beiglb\"{o}ck et al~\cite{BeiglbockHenriLaborderePenkner:13} the problem of constructing martingales with given laws at two times
(whether via a connection with the {\sc Sep} or otherwise) is sometimes called the
martingale optimal transport problem ({\sc Mot}). Given a pair $(\mu,\nu)$ of measures which are increasing in convex order, the {\sc Mot} is to construct a
coupling or joint law with marginals $(\mu,\nu)$ which respects the martingale property. {\sc Mot} theory has proved extremely useful both in constructing
extremal consistent models, and in proving martingale inequalities (see e.g. Acciaio et al~\cite{AcciaioBeiglbockPenknerSchachermayerTemme:14}). Again, from
{\sc Mot} based on pairs of measures, there is a natural extension to finite families of measures.

The problem of constructing a fake version of a continuous-time martingale can be thought of as the
infinite-marginal generalisation of the multi-marginal {\sc Sep} or
{\sc Mot}. Indeed, the existence of solutions of these problems underpins a natural approach to the mimicking problem: discretize time
using a sequence of partitions to reduce the
continuous-time problem based on $(\mu_t)_{0 \leq t \leq T}$ to a sequence of discrete-time problems
with marginals $(\tilde{\mu}^{(n)}_{k})_{0 \leq k \leq K(n)}$
(where $\tilde{\mu}^{(n)}_k = \mu_{kT/K(n)}$);
construct a sequence of discrete-time martingales ${\tilde{X}}^{(n)}$ with marginals $(\tilde{\mu}^{(n)}_{k})_{0 \leq k \leq K(n)}$;
re-interpret these discrete-time
processes as skeletons of piecewise constant continuous-time martingales $X^{(n)}$; and finally show that the processes
$X^{(n)}$ converge to a consistent martingale. It is this last step
which is potentially difficult, especially if we want to show that some optimality property is preserved in the limit. This is the approach taken in both
Henri-Labord\`{e}re et al~\cite{Henri-LabordereTanTouzi:15} and K\"{a}llblad et al~\cite{KallbladTanTouzi:15} based on solutions of the {\sc Sep/Mot}
due to Beiglb\"ock and Juillet~\cite{BeiglbockJuillet:14} and Henri-Labord\`{e}re et al~\cite{Henri-LabordereOblojSpoidaTouzi:15} respectively.

The approach in this article is different. Rather than using an approximation technique we aim to write down the mimicking martingale directly. Conceptually,
this is more direct than the approximation procedure used in \cite{Henri-LabordereTanTouzi:15} and
\cite{KallbladTanTouzi:15}. There is one key idea which allows us to do this: our solutions are based on the principle that whenever possible we leave
the location of mass unchanged. The processes we construct have finitely many jumps in every interval $[\epsilon,T]$ and are constant between
jumps.

Our main contributions are as follows. Firstly, under fairly weak assumptions we give a family of martingales whose univariate marginals match a given set
of measures. Secondly, under an extra simplifying assumption - the extension of the `dispersion assumption' in \cite{HobsonKlimmek:15} from pairs of marginals to
a continuum of marginals - we identify the consistent martingale with smallest expected total variation.
(Note that if our goal is to mimic a martingale
diffusion $M$ then necessarily $M$ has infinite variation, almost surely.) Thirdly, we give a pathwise inequality which can be interpreted as
a model independent sub-hedge in the exotic option pricing setting.

In particular, we show that there is a
martingale with Gausian univariate marginals (mean zero, variance $t$ at time $t$) which unlike Brownian motion has finite expected total variation. More generally,
suppose $(\mu_t)_{0 \leq t \leq T}$ is a family of probability measures which is increasing in convex order, and suppose that $\mu_t$ has a continuous density
$\rho(t,x)$ and is such that $\rho$ is differentiable in $t$ (and such that $\rho$ and $\dot{\rho}$ have certain integrability properties).
Now add the hypothesis
that $\dot{\rho}(t,y) \leq K(t) \rho(t,y)$ for some decreasing function $K(t)$: in the Brownian case we can take $K(t)=1/2t$.
In this setting we show how to construct a consistent martingale from a Poisson point process and a family of martingale couplings.
Different martingale couplings will lead to different
mimicking processes, so this construction potentially yields many consistent martingales.
Further, since the dispersion assumption is satisfied in the Brownian case, we can construct the fake Brownian motion of minimal expected total variation.

\section{Construction of a family of mimicking process}
\label{sec:candidate}

\subsection{Set-up, assumptions and examples}
\label{ssec:setup}
Fix $\epsilon, T$ with $0 \leq \epsilon<T<\infty$ and let $\bbT$ be the open interval $\bbT = (\epsilon , T)$ and $\cbbT = [\epsilon, T]$. We suppose we are given
$(\mu_t)_{t \in \cbbT}$; our aim is to construct a
martingale $X = (X_t)_{t \in \cbbT}$ with marginals $(\mu_t)_{t \in \cbbT}$.

We begin by stating our assumptions on the family of measures $(\mu_t)_{t \in \cbbT}$. 
The first assumption is necessary for the family
$(\mu_t)_{t \in \cbbT}$ to be increasing in convex order, and hence by Kellerer's Theorem~\cite{Kellerer:72} for there to exist a
martingale with this family of marginal distributions.
A key quantity is $U(t,x) := \int |y-x| \mu_t(dy)$ which is (minus) the
potential of
$\mu_t$. Then $U$ is non-decreasing in $t$ for each fixed $x$ if and only if the
family $(\mu_t)_{t \in \cbbT}$ is increasing in convex order.

\begin{ass}
\label{ass:1}
\begin{enumerate}
\item[(a)] For all $t \in \cbbT$, $\int \mu_t (dy) = 1$; $\int |y| \mu_t(dy) < \infty$; $\int y \mu_t(dy) = \overline{\mu}$.
\item[(b)] $U(t,x) := \int |y-x| \mu_t(dy)$ is non-decreasing in $t \in \cbbT$ for each $x$.
\end{enumerate}
\end{ass}


Now we impose some further regularity conditions on the potential.
\begin{ass}
\label{ass:2}
\begin{enumerate}
\item[(a)] The potential $U(t,x)$ is continuous in $t$ for $t \in \cbbT$.

\item[(b)] $Q_t(x) = Q(t,x) := \frac{1}{2}\dot{U}(t,x)$
exists for each $t \in \bbT$ and is continuous in $x$ and $t$.
Also, $Q_t(x)$ satisfies $\lim_{x \rightarrow \pm \infty} Q_t(x)=0$ and $\lim_{x \rightarrow \pm \infty} Q_t(x) - xQ'_t(x)=0$.
Since $U$ is non-decreasing in $t$ we have $Q_t \geq 0$.

\item[(c)] For each $t \in \bbT$, $Q_t(\cdot)$ is the difference of two positive, decreasing convex functions. In particular,
$Q_t(x) = \Lambda_t(x) - \Gamma_t(x)$ where
$\lim_{x \uparrow \infty} \Lambda_t(x)=0$.
Further, $\Lambda_t$ is such that $m(t):=\lim_{x \downarrow -\infty} \Lambda_t(x)/|x|$ and
$n(t) :=  \lim_{x \downarrow -\infty} \Lambda_t(x) - x\Lambda_t'(x)$
are finite. Finally, $\int_{t \in \bbT} m(t) dt < \infty$.
\end{enumerate}

Since $\Lambda_t$ and $\Gamma_t$ are convex we can write $\lambda_t = \Lambda_t''$ and $\gamma_t= \Gamma''_t$, whence $\gamma_t$ and $\lambda_t$ are
measures (which we argue below have the same finite total mass $m(t)$). For uniqueness we insist that $\lambda_t$
and $\gamma_t$ are orthogonal. This is equivalent to choosing the decomposition of $Q_t = \Lambda_t - \Gamma_t$ such that $\Lambda_t$ and $\Gamma_t$ are
as small as possible.
Then $q_t = Q_t''$ may be interpreted as a signed measure and $q_t(dx) = \lambda_t(dx) - \gamma_t(dx)$.

\begin{enumerate}
\item[(d)] For each $t \in \bbT$, $\gamma_t$ is absoluteluy continuous with respect to $\mu_t$
and
the Radon-Nykodym derivative
$R_t(x) = \gamma_t(dx)/\mu_t(dx)$ is measurable in $x$ and $t$ and is bounded in the sense that $R_t(x) \leq K(t)$ for some function $K(t)$
which is bounded on $\bbT$ by $\overline{K}$.

\item[(e)] For any measurable set $B$ and for all $t \in \cbbT$, $\mu_t(B) = \mu_{\epsilon}(B) + \int_{\epsilon}^t q_s(B) ds$.

\end{enumerate}
\end{ass}



Fix $t \in \bbT$. By assumption $\lim_{x \rightarrow \pm\infty} Q_t(x) = 0$ and since $Q_t$ is the difference of convex functions and since
$\lim_{x \uparrow \infty} \Lambda_t(x)=0$ we have $\lim_{x \uparrow \infty} \Gamma_t(x)=0$ also.
The fact that $\lim_{x \uparrow \infty} \Lambda_t(x) - x \Lambda_t'(x)=0$ follows from the convexity of $\Lambda_t$ as does the existence of
$\lim_{x \downarrow -\infty} \Lambda_t(x) - x \Lambda_t'(x)$ in $[0,\infty]$. Hence the content of the assumption on $n(t)$ is that $n(t)$ is finite.

Convexity of $\Lambda_t$ implies that $\lim_{x \downarrow -\infty} \Lambda_t(x)/|x|$ exists in $[0,\infty]$, so again, the content of the assumption on $m(t)$
is that this limit is finite.
In addition $\lim_{x \downarrow -\infty} Q_t(x)=0$ by hypothesis and we conclude that $\lim_{x \downarrow -\infty} \Gamma_t(x)/|x| = m(t)$ also.
In particular, $\gamma_t$ and $\lambda_t$ have the same total mass, and $q_t$ has zero total mass.

Finally, since $\lim_{x \downarrow -\infty} Q_t(x) - x Q_t'(x)=0$ and $\lim_{x \downarrow -\infty} \Lambda_t(x) - x \Lambda_t'(x)=n(t)$,
we conclude that $\lim_{x \downarrow -\infty} \Gamma_t(x) - x \Gamma_t'(x)=n(t)$ also. Hence $\lambda_t$ and $\gamma_t$ have the same mean. The fact that
$Q_t \geq 0$ implies that the pair $(\gamma_t,\lambda_t)$ is increasing in convex order.

\begin{rem}{\rm
Our ultimate goal is to construct a process on $[0,T]$. However, it is clear from the examples that below that
often there is no universal bound $\overline{K}$ for $K(t)$ on intervals of the form
$\bbT = (0,T)$, and the processes
we build will have infinitely many jumps on such intervals. To circumvent any problems related to the presence of infinitely many jumps
near $t=0$ we first construct processes with time parameter
set $\cbbT = [\epsilon, T]$, and then extend to larger intervals such as $[0,T]$ by concatenation. See Section~\ref{sec:extension}.
}\end{rem}

\begin{rem}{\rm
The requirement that $Q$ exists for all $t \in \bbT$ can be relaxed to an assumption that $Q$ exists and has the properties of Parts (b), (c) and (d) of
Assumption~\ref{ass:2} on a subset $\bbT_0 \subseteq \cbbT$ of full measure. }
\end{rem}

For each $t \in \cbbT$, let $I^{\mu}_t$ be the smallest interval such that $\mu_t$ has support in $I^{\mu}_t$.
Let the endpoints of $I^\mu_t$ be given as $\{ {\ell}_{\mu}(t), r_{\mu}(t) \}$. Let $I = \cup_{t \in \cbbT} I^\mu_t$.
Similarly, let $I^\lambda_t$ with endpoints $\{ \ell_\lambda(t), r_\lambda(t) \}$ denote the smallest interval such that
$\lambda_t$ has support in $I^\lambda_t$, with similar conventions
for other measures. Then, since the pair $(\gamma_t,\lambda_t)$ is increasing in convex order we must have
\( \ell_{\mu}(t) \leq \ell_{\lambda}(t) \leq \ell_{\gamma}(t) \leq r_{\gamma}(t) \leq r_{\lambda}(t) \leq r_{\mu}(t)$.

\begin{defn}
\label{def:regular}
By the regular case we mean that $U(t,x) \in C^{1,2}$ (so that the density $\rho(t,x)$ of $\mu_t$ and $\dot{\rho}(t,x)$ exist and are continuous),
that for each $t \in \bbT$, $\int |\dot{\rho}(t,x)| dx < \infty$ and $\int |x| |\dot{\rho}(t,x)|dx < \infty$ and also
$\int_{\bbT} dt \int |\dot{\rho}(t,x)|dx < \infty$.
\end{defn}

\begin{rem}{\rm
Suppose we are in the regular case. Then
$\Gamma_t(x) = \int_x^\infty (y-x) \dot{\rho}(t,x)^- dx$ and $\Lambda_t(x) = \int_x^\infty (y-x) \dot{\rho}(t,x)^+ dx$. Since
$\int |\dot{\rho}(t,x)| dx < \infty$ and $\int |x| |\dot{\rho}(t,x)|dx < \infty$ it follows that $m(t)$ and $n(t)$ are well
defined and finite. Then
\[ m(t) = \lim_{x \downarrow - \infty} \frac{\Gamma_t(x)}{|x|} = \lim_{x \downarrow - \infty} \left\{ \int_x^\infty
\frac{y \dot{\rho}(t,y)^-}{|x|} dy + \int_x^\infty \dot{\rho}(t,y)^- dy \right\} = \int_{-\infty}^\infty \dot{\rho}(t,y)^- dy \]
and $n(t) = \int_{-\infty}^\infty y \dot{\rho}(t,y)^- dy$.


It follows from the condition that $\int_x |\dot{\rho}(t,x)| dx$ is integrable on $\bbT$ and an application of Fubini's theorem that
\( 0 = \int_{-\infty}^\infty \dot{\rho}(s,x)dx ,\) and hence that $m(t) = \int \dot{\rho}(t,x)^- dx = \int \dot{\rho}(t,x)^+ dx$.
Similarly we find that
$n(t) = \int x \dot{\rho}(t,x)^- dx = \int x \dot{\rho}(t,x)^+ dx$.
If $H$ is bounded and measurable, we have
\[ \int_{t_0}^t \! du \! \int H(x) q_u(dx) = \int \! H(x) \int_{t_0}^t \! \dot{\rho}(u,x)dx
= \int \! H(x) [\rho(t,x) - \rho(t_0,x)] dx = \int \! H(x) \mu_t(dx) - \int \! H(x) \mu_{t_0}(dx) \]
and taking $H(x) = I_{ \{ x \in B \} }$ we find
\( \frac{d}{dt} \mu_t(B) = q_t(B). \)
Finally, in the regular case, $R_t(x)= \dot{\rho}(t,x)^-/\rho(t,x)$, and Assumption {\rm \ref{ass:2}(d)} becomes
$\dot{\rho}(t,x)^- \leq K(t) \rho(t,x)$ uniformly in $x \in I$, for some decreasing function ${K}$.
In particular, in the regular case verifying that the conditions of Assumption~\ref{ass:2} are satisfied reduces to verifying that
 $\dot{\rho}(t,x)^- \leq \rho(t,x) K(t)$.

}\end{rem}

We give illustrative five examples. The first two examples are natural examples where the situation is regular in the sense of Definition~\ref{def:regular}.
The third example is also completely natural, but is such that although $\mu_t$ and $\gamma_t$ have nice densities the measure $\lambda_t$ is atomic.
The fourth example has a reverse structure in which $\gamma_t$ is purely atomic (and $\mu_t$ has an atom), but $\lambda_t$ has a density.
The final example is beyond the
scope of our methods: the issue is not that $\mu_t$ has atoms, but rather that the atoms of $\mu_t$ move over time.

Let $\delta_{y}$ denote the point mass at $y$ so that $\delta_y(B) = I_{ \{ y \in B \} }$.
Define $\Phi(x) = \int_{-\infty}^x \frac{e^{-y^2/2}}{\sqrt{2 \pi}} dy$.

\begin{eg}[Brownian motion]
\label{eg:bm} {\rm Let $\mu_t$ be
the law of a centred Gaussian variable with variance $t$. Then the family $(\mu_t)_{t \geq 0}$ is the family of marginals of Brownian motion. We have
\begin{eqnarray*}
U(t,x) & = & 2 \sqrt{t} \left\{ \frac{e^{-x^2/2t}}{\sqrt{2 \pi}} + \frac{x}{\sqrt{t}} {\Phi}\left( \frac{x}{\sqrt{t}} \right) - \frac{x}{2 \sqrt{t}} \right\};
   \\
\rho(t,x) = \frac{1}{2} U''(t,x) & = & \frac{1}{\sqrt{2 \pi t}} e^{-x^2/2t}; \\
Q(t,x) = \frac{1}{2} \dot{U}(t,x) & = & \frac{1}{2 \sqrt{t}} \frac{e^{-x^2/2t}}{\sqrt{2 \pi}}; \\
\dot{\rho}(t,x) & = & \frac{\rho(t,x)}{2 t^2} (x^2 - t); \\
R_t(x) = \frac{(t-x^2)^+}{2t^2} & \leq &  \frac{1}{2t} =: K(t) .
\end{eqnarray*}
For this example $m(t)= \frac{1}{2t} \int_{-1}^1 (1-y^2) e^{-y^2/2} \frac{dy}{\sqrt{2\pi}} = \frac{1}{t}\frac{e^{-1/2}}{\sqrt{2 \pi}}$ and $n(t)=0$.
}\end{eg}

\begin{eg}[Exponential Brownian motion]
\label{eg:ebm}
{\rm Let $\mu_t$ be the time-$t$ law of exponential Brownian motion.
\begin{eqnarray*}
U(t,x) & = &  2 \Phi \left( \frac{- \ln x}{\sqrt{t}} + \frac{\sqrt{t}}{2} \right) - 2 x \Phi \left( \frac{- \ln x}{\sqrt{t}} - \frac{\sqrt{t}}{2} \right) + x; \\
\rho(t,x) = \frac{1}{2} U''(t,x) & = & \frac{1}{x^{3/2} \sqrt{2 \pi t}} e^{-(\ln x)^2/2t}
e^{-t/8}; \\
Q(t,x) = \frac{1}{2} \dot{U}(t,x) & = &  \frac{\sqrt{x}}{2 \sqrt{2 \pi t}} e^{-(\ln x)^2/2t}
e^{-t/8} = \frac{x^2}{2} \rho(t,x); \\
\dot{\rho}(t,x) & = & \frac{\rho(t,x)}{2 t^2} [(\ln x)^2 - t
- t^2/4 ]; \\
R_t(x) = \left[ \frac{1}{8} + \frac{1}{2t} - \frac{(\ln x)^2}{2 t^2} \right]^+ & \leq &
1/(2t) + 1/8 =: K(t).
\end{eqnarray*}
Note that this example is asymmetric about the mean 1 and $n(t) \neq 1$.

In general there are few examples of fake exponential Brownian motions --- a rare exception is \cite{Hobson:13}.
The lognormal law does not have the scaling properties of the normal distribution, and as a consequence it seems harder to construct
elegant mimicking processes.}
\end{eg}

\begin{eg}[Continuous uniform]
\label{eg:uniform}
{\rm Let $\mu_t \sim  U[-t,t]$. Although $\gamma_t$ has a density, $\lambda_t$ consists of a pair of point masses.
\begin{eqnarray*}
U(t,x) & = & \frac{(t^2 + x^2)}{2t} I_{\{ -t < x < t \} } + |x| I_{ \{ |x| \geq  t \} } ;    \\
\rho(t,x) = \frac{1}{2} U''(t,x) & = & \frac{1}{2t} I_{\{ -t < x < t \} } ; \\
Q(t,x) = \frac{1}{2} \dot{U}(t,x) & = & \frac{1}{4t^2} (t^2 - x^2) I_{\{ -t < x < t \} } ; \\
\gamma_t(dx) & = & \frac{dx}{2t^2} I_{\{ -t < x < t \} } ; \\
\lambda_t & = & \frac{1}{2t} \delta_{-t} + \frac{1}{2t} \delta_{t} ; \\
R_t(x) = \frac{1}{t} I_{\{ -t < x < t \} } & \leq &  \frac{1}{t} := K(t) .
\end{eqnarray*}
In this case $\gamma_t = \frac{1}{t}U[-t,t]$ and $\lambda_t =\frac{1}{t} U \{ -t,t \}$ so that the signed measure $q_t$
does not have a density.
Nonetheless, for $B = (\underline{b}, \overline{b})$ with $0 < \underline{b} < \overline{b}$ (for example) we find
$\mu_t(B) = \frac{ \overline{b} \wedge t -\underline{b} \wedge t}{2t}$  and $\frac{d}{dt} \mu_t(B) =
\frac{1}{2t} I_{ \{ \underline{b} < t < \overline{b} \} } -\frac{ \overline{b} \wedge t -\underline{b} \wedge t}{2t^2} = q_t(B)$
for all $t$ except $t = \underline{b}$ and $t= \overline{b}$. Hence $\mu_t ((\underline{b}, \overline{b})) = \mu_{\epsilon}((\underline{b}, \overline{b}))
+ \int_\epsilon^t q_s((\underline{b}, \overline{b})) ds$.
}\end{eg}

\begin{eg}
\label{eg:mass0}
{\rm Suppose $\mu_t \sim e^{-t} \delta_0 + (1-e^{-t}) U[-1,1]$.
\begin{eqnarray*}
U(t,x) & = & \left[ \frac{(1 + x^2)}{2}(1 - e^{-t}) + e^{-t}|x| \right] I_{\{ -1 < x < 1 \} } + |x| I_{ \{ |x| \geq  1 \} } ;    \\
Q(t,x) = \frac{1}{2} \dot{U}(t,x) & = & \frac{e^{-t}}{4} (1 + x^2 - 2|x|)) I_{\{ -1 < x < 1 \} } ; \\
\gamma_t & = & e^{-t} \delta_0  ; \\
\lambda_t(dx) & = & \frac{e^{-t}}{2} I_{\{ -1 < x < 1 \} } dx  ; \\
R_t(x)  = I_{\{ x=0 \}} & \leq & 1 := K(t) .
\end{eqnarray*}
}\end{eg}

Note that in the first three of these examples if $\bbT =(\epsilon,T)$ with $\epsilon >0$ then we may take $\overline{K} = K(\epsilon)$, but if $\epsilon=0$
then there is no global bound on $R_t(x)$.

\begin{eg}[Two point Discrete Uniform]
{\rm Suppose $\mu_t \sim \frac{1}{2} \delta_{ t } + \frac{1}{2}
\delta_{  -t  }$. Then
$U(t,x)  = \max \{ t, |x| \}$ and $Q(t,x) = I_{\{ -t < x < t \} }$. This indicator function cannot be written as the difference
of two convex functions, and this example lies out the scope of our analysis. Note, it is not the presence of atoms in this example which means that
our approach does not work, but rather the fact that the locations of the atoms move over time. }
\end{eg}

\subsection{The construction}
\label{ssec:construction}
Our goal is to construct a pure-jump Markov martingale with the required marginals.
The idea behind the process $X$ we construct is that it stays
where it is whenever and wherever
possible. In Section~\ref{sec:minV} this will be one of the features which leads to $X$ having
minimal variation. It also leads to a simple construction:
most especially we can write down the process directly, and no limiting procedures are used.


The construction is based upon the twin ingredients of a Poisson point process and a
family of martingale transports. The Poisson point process determines whether there is a jump,
and then conditional on there being a jump the martingale transports determine the new location of the process.
  Our problem is to specify
the ingredients in such a way that the resulting jump process has the desired marginals.

First consider the point process.
Let $(N_t)_{t > 0}$ be a Poisson point process on $(0,\infty) \times
(0,\infty)
\times (0,1)$ with density $ds \times dh \times du$. Let $\sN$ be the
set of locations of events of the Poisson process. We think of the three
coordinates as time, height and label. If we want to restrict the time domain of $N$ to an interval $A$ e.g. $A=(\epsilon,T]$ then we write
$N^{A}$ for the Poisson point process $(N_t)_{t \in A}$ with density $ds \times dh \times du$ on
$A \times (0,\infty) \times (0,1)$.

Conditional on there being an event of the Poisson Process at $(s,h,u)$ for
which $R_s(X_s)<h$ then the process jumps at time $s$. The location of the jump destination depends on the
label $u$ and the martingale transport.
If the process is at $x$ at time $s$, if the first two coordinates of $N$ determine that there should be a jump, and if $\pi^x_s$
denotes the regular conditional distribution of
the post-jump location, then we set the post-jump location to be $F^{-1}_{\pi^x_s}(u)$ where $u$ is the label and $F_{\pi^x_s}(y) = \pi^x_s((-\infty, y])$.
By convention (this will not impact on whether $X$ has the required marginals) at the
time of the jump $X$ takes the value it jumps to rather than the value it jumps from. Hence, the process $X$ we construct
is c\'{a}dl\'{a}g.

Now consider the family of martingale transports.
Suppose $\nu_0$ and $\nu_1$ are probability measures on $I$.
Let $\sM_1(\nu_0,\nu_1)$ be the set of martingale couplings or martingale transports from $\nu_0$ to
$\nu_1$. In particular, $\pi \in \sM_1(\nu_0,\nu_1)$ is a probability measure on $I \times I$ which satisfies the marginal constraints
$\int_y \pi(dx,dy) = \nu_0(dx)$ and $\int_x \pi(dx,dy) = \nu_1(dy)$ and the martingale constraint
$\int_y(y-x) \pi(dx,dy) = 0$. Then $\pi$ has an interpretation as a rule for the `transport' of mass from $\nu_0$ to $\nu_1$
which respects the martingale property.
We have that $\sM_1(\nu_0,\nu_1)$ is non-empty if and only
if $\nu_0$ and $\nu_1$ are increasing in convex order which we now assume.
Any $\pi \in \sM_1(\nu_0,\nu_1)$ can be written in disintegration form
$\pi(dx,dy) = \pi^x(dy) \nu_0(dx)$. Then $\int_y \pi^x(dy)=1$, $\int_y y \pi^x(dy) = x$ and $\int_x \pi^x(dy) \nu_0(dx) = \nu_1(dy)$.
If $\nu_0$ and $\nu_1$ are orthogonal then $\pi^x(\{x\})=0$.

The notion of a martingale transport generalises easily to any pair $(\tilde{\nu}_0,\tilde{\nu}_1)$
of measures with the same finite total mass which are in convex order.
Let $m$ be the mass of $\tilde{\nu}_0$.
If $\hat{\nu}_i = \tilde{\nu}_i/m$ and if $\hat{\pi} \in \sM_1(\hat{\nu}_0, \hat{\nu}_1)$ has disintegration representation
$\hat{\pi}(dx,dy) = \hat{\nu}_0(dx)\pi^x(dy) $ then $\tilde{\pi}$ given by
$\tilde{\pi}(dx,dy) = m \hat{\pi}(dx,dy) = m \hat{\nu}_0(dx)\pi^x(dy)  = \tilde{\nu}_0(dx) \pi^x(dy) $
is a martingale transport of $\tilde{\nu}_0$ to $\tilde{\nu}_1$.
Let $\sM = \sM(\tilde{\nu}_0, \tilde{\nu}_1)$ denote the set of martingale transports in this wider sense.

Several martingale couplings are known and can be described in more-or-less explicit terms. Many of the couplings are associated with
solutions of a Skorokhod embedding problem for
Brownian motion for a non-trivial initial law. For a pair of probability measures in increasing convex order,
Chacon and Walsh~\cite{ChaconWalsh:76} give a general family of solutions, of which the construction in Hobson~\cite{Hobson:98} is a special case related
to the Az\'ema-Yor~\cite{AzemaYor:79} embedding (the Az\'ema-Yor embedding covers the case of an initial law which is a point mass). The Hobson~\cite{Hobson:98} solution
inherits from the Az\'ema-Yor embedding the property of having a very explicit construction.
Bertoin and Le Jan~\cite{BertoinLeJan:92} give a martingale coupling based on local times
in the case where $\nu_0$ and $\nu_1$ are orthogonal. Hobson and Neuberger~\cite{HobsonNeuberger:12} gave the coupling which maximises
the $L^1$ norm of the increment.
Recently, Beiglb\"ock and Juillet~\cite{BeiglbockJuillet:14} introduced the `left-curtain' martingale coupling.
For our purposes the construction of Hobson and Pedersen~\cite{HobsonPedersen:02} has certain desirable properties in that
the first step of the construction is to fix any mass in common between $\nu_0$ and $\nu_1$, consistent with the general idea
of leaving mass unmoved wherever possible. The Hobson-Pedersen construction is one of the most explicit martingale coupling for general measures, and we describe
it in detail in Section~\ref{ssec:HPconstruction} below.
In Section~\ref{sec:minV} we will exploit the coupling of Hobson and Klimmek~\cite{HobsonKlimmek:15}
which like the constructions of \cite{HobsonNeuberger:12} and \cite{BeiglbockJuillet:14} has certain optimality properties.

Suppose we have our Poisson point process and a family of martingale transports $(\pi_t)_{t \in \bbT}$ with $\pi_t \in \sM(\gamma_t,\lambda_t)$.
We begin by explaining, first informally and in the regular case, why the process $X=(X_t)_{t \in \cbbT}$ we construct has density $\rho(t,x)$.

Suppose that at time $v \in \bbT$ the process is at $x$ and $x$ is in the support of $\gamma_v$ and more generally that $x$ is in the support of $\gamma_s$
for $v \leq s \leq T$. Then $(X_t)_{v
\leq t}$
stays constant until $\tau  = \inf \{ s > v ; \exists (s,h,u) \in \sN \cap (h
\leq R_s(x)) \}$. Note that since $R_s(x) < K(s) \leq \overline{K}$ we have
that $\tau > v$ almost surely.
Then, from the properties of the Poisson process, for fixed $v$, $t$ and $x$,
\[ \Prob^{v,x} \left( \{
\sN \cap (v \leq s \leq t) \cap (h \leq R_s(x))
\} = \emptyset \right)
= \exp \left( - \int_v^t R_s(x) ds \right) =
\exp \left( \int_v^t \frac{\dot{\rho}(s,x)}{\rho(s,x)}  ds \right)
= \frac{\rho(t,x)}{\rho(v,x)} , \]
and it follows that
\[ \Prob(X_t \in dx) = \Prob(X_v \in dx) \frac{\rho(t,x)}{\rho(v,x)} =
\rho(t,x) dx .\]

Now suppose that $x$ is not in the support of $\gamma_v$. Consider the rate of change in density at $x$
at time $v$, assuming that $X_v$ has density $\rho(v,x)$. We have
\begin{equation}
\label{eq:heuristic}
\Prob(X_{v + dv} \in dx) = \Prob(X_v \in dx) +
\int_{z \in I^\gamma_v} \Prob(X_v \in dz)  dv  \; R_v(z) \pi^z_v(dx).
\end{equation}
If $\pi_t \in \sM(\gamma_t,\lambda_t)$ then the right-hand-side of (\ref{eq:heuristic}) can be rewritten as
\begin{eqnarray*}
 \rho(v,x) dx + \int_{z \in I^\gamma_v} \rho(v,z)R_v(z) dz \pi_v^z(dx) dv
& = & \rho(v,x) dx + \int_{z \in I^\gamma_v} \gamma_v(dz) \pi^z_v(dx)  dv \\
& = & \rho(v,x) dx + \lambda_v(dx) dv \\
& = & \rho(v,x)dx + \dot{\rho}(v,x) dx \; dv ,
\end{eqnarray*}
and hence $\Prob(X_{v+dv} \in dx) = \rho(v+dv,x) dx$ as required.


Our goal now is to extend this heuristic argument to prove that $X_t$ has law $\mu_t$.
For $t \in \bbT$ define the operator $\sA^{\pi_t}_t$ by
\[ \sA_t^{\pi_t} h(x) = R_t(x) \left\{ \int_y \pi_t^x(dy) h(y) - h(x) \right\} \]
and recall that this is non-zero only if $x \in \supp( \gamma_t)$.
Suppose $(\pi_t)_{t \in \bbT}$ is chosen with $\pi_t \in \sM(\gamma_t,\lambda_t)$ and let $X = (X_t)_{t \in \bbT}$ be the
Markov process with generator $\sA^{\pi_t}_t$ and initial law $X_\epsilon \sim \mu_\epsilon$. Then $X$ is a pure-jump Markov process.

\begin{ass}
\label{ass:3}
Suppose the family $(\pi_t)_{t \in \bbT}$ is such that
$\pi_t \in \sM(\gamma_t,\lambda_t)$ has representation $\pi_t(dx,dy) = \gamma_t(dx) \pi_t^x(dy)$ and
for all measurable sets $B$, $\pi^x_t(B)$ is measurable in $x$ and $t$.
\end{ass}

Our main theorem says that under the assumptions of this section, the process $X = (X_t)_{t \in \cbbT}$ has marginals $(\mu_t)_{t \in \cbbT}$.
The proof of this result relies on results of Feller~\cite{Feller:40} and Feinberg et al~\cite{FeinbergMandavaShiryaev:13}
on the existence and uniqueness of solutions of the Kolmogorov forward equations for pure-jump Markov processes.
Feller~\cite{Feller:40} proves existence under an assumption that jump rates are continuous in $t$, and this condition is weakened to measurability in
\cite{FeinbergMandavaShiryaev:13}.
The idea in these papers is to use the fact that there are only finitely many jumps (almost surely, in $\bbT$) to construct the process and its transition functions
as a limit of processes which have at most $n$ jumps. These pure-jump proceses with a bounded number of jumps are simple to characterise, and it is possible to
write down the probability transition functions in terms of integrals involving the transition rates.

\begin{thm}
\label{thm:mimic}
Suppose Assumptions~\ref{ass:1}, \ref{ass:2} and \ref{ass:3} hold. Suppose $\bbT = (\epsilon,T)$ and suppose
$X_\epsilon \sim \mu_\epsilon$. Then for $t \in \cbbT$, $X_t \sim \mu_t$.
\end{thm}

\begin{proof}
For $t \in \cbbT$, let $\nu_t$ denote the law of $X_t$. Then $\nu_\epsilon = \mu_\epsilon$. Further, since $R_t(x)$ is bounded, $I$ is $q$-bounded in the language of
Feinberg et al~\cite{FeinbergMandavaShiryaev:13},  and by Feller~\cite[Equation (37)]{Feller:40} or
Feinberg et al~\cite[Corollary 4.2]{FeinbergMandavaShiryaev:13},
for measurable sets $B$,
\begin{equation}
\label{eq:KFEnu}
 \nu_t(B) = \mu_\epsilon(B) + \int_\epsilon^t ds \int_{x \in I} \nu_s(dx) R_s(x) \pi^x_s(B) -  \int_\epsilon^t \int_B R_s(x) \nu_s (dx).
 \end{equation}
Conversely, by construction and design the family $(\mu_t)_{t \in \cbbT}$ is such that
\begin{eqnarray}
 \mu_t(B) & = & \mu_\epsilon(B) + \int_\epsilon^t ds \; q_s(B) \nonumber \\
 &  = & \mu_\epsilon(B) + \int_\epsilon^t ds \int_B \lambda_s(dy) - \int_\epsilon^t ds \int_B \gamma_s(dx) \nonumber \\
 &  = & \mu_\epsilon(B) + \int_\epsilon^t ds \int_{x \in I} \int_{y \in B} \gamma_s(dx) \pi^x_s(dy) - \int_\epsilon^t ds \nonumber \int_B \gamma_s(dx) \\
 & = & \mu_\epsilon(B) + \int_\epsilon^t ds \int_{x \in I} \mu_s(dx)  R_s(x) \pi^x_s(B) - \int_\epsilon^t ds \int_B R_s(x) \mu_s(dx) .
\label{eq:KFEmu}
\end{eqnarray}
Let $(\eta_t)_{t \in \cbbT}$ be the family of signed measures given by $\eta_t = \nu_t - \mu_t$,
and note that $\eta_t(I)=0$. Let $\Delta_t$ denote the total variation distance between $\nu_t$ and
$\mu_t$ so that $\Delta_t = \sup_{B} |\eta_t(B)|$, and let
$\overline{\Delta}_t = \sup_{\epsilon \leq s \leq t} \Delta_s$.
Then, comparing (\ref{eq:KFEnu}) and (\ref{eq:KFEmu})
\[ \eta_t(B) =  \int_\epsilon^t ds \: \int_{x \in I} \eta_s(dx)  R_s(x) \left[ \int_{x \in I} \pi^x_t(B) - I_{ \{ x \in B \}} \right]  \]
and it follows that
\[ |\eta_t(B)| \leq  \int_\epsilon^t ds \int_{x \in I} | \eta_s(dx) | K(s) \leq  2  \int_\epsilon^t K(s) \Delta_s ds
\leq 2 \overline{K} \int_{\epsilon}^t \overline{\Delta}_s ds . \]
Hence $\Delta_t \leq 2 \overline{K}\int_\epsilon^t \overline{\Delta}_s ds$ and since the right-hand-side of this last expression is increasing,
$\overline{\Delta}_t \leq  2 \overline{K} \int_\epsilon^t \overline{\Delta}_s ds$. It follows from Gronwall's Lemma
(Ethier and Kurtz~\cite[Appendix, Theorem 5.1]{EthierKurtz:84})
that $\Delta_t = \overline{\Delta}_t = 0$ and $X_t \sim \mu_t$ for $t \in \bbT$. Finally, since the rate of jumps is bounded,
$\Prob(X_T \in B) = \lim_{s \uparrow T}\Prob(X_s \in B) = \lim_{s \uparrow T} \mu_s(B) = \mu_T(B)$ by Assumption~\ref{ass:2}(e).
\end{proof}

\subsection{An explicit martingale transport}
\label{ssec:HPconstruction}

The goal of this section is to show that even when $\gamma_t$ and $\lambda_t$ are arbitrary (integrable) measures we can write down a
martingale coupling.

Fix $t$ and write $\gamma$ as shorthand for $\gamma_t$, and similarly $\lambda$ as shorthand for $\lambda_t$, where the orthogonal measures
$\gamma$ and $\lambda$
have the same finite total mass and the pair $(\gamma,\lambda)$ is increasing in convex order. Our aim is to write down a
martingale transport $\pi(dx,dy)$ which has marginals $\gamma$ and $\lambda$ and respects the martingale condition. In fact we write down $\pi^x(dy)$; then
$\pi(dx,dy) = \gamma(dx) \pi^x(dy)$. The construction we give is based on the solution to the Skorokhod embedding problem with
non-trivial initial law given in Hobson
and Pedersen~\cite{HobsonPedersen:02}. The only difference with respect to that paper is that there $\gamma$ and $\lambda$ are probability measures,
but the extension to measures with finite mass does not introduce major complications.

Let $Q(x) = \frac{1}{2} \{ \E^{X \sim \lambda}[|X- x|] - \E^{X \sim \gamma}[|X- x|]\}$. Since $\gamma$ and $\lambda$ are in convex order we have $Q \geq 0$, and
$Q$ is continuous in $x$. Define
\[ \xi(z) = \sup_{w: w < z} \frac{Q(z) - Q(w)}{z-w} \]
and let $g(z) \leq z$ be the largest of those values where the supremum is attained, or if the supremum is not attained then $\xi(z) = Q'(z-)$ and we set
$g(z)=z$. Now define
\( \Xi(z) = \lambda( (-\infty, z)) - \xi(z). \)
Then (Hobson and Pedersen~\cite[Theorem 2.1]{HobsonPedersen:02}) $\Xi$ is an increasing, left-continuous function.

Decompose $\lambda$ into a continuous part and a sum of atoms: $\lambda(A) = \int_{z \in A} \lambda^c(dz) + \sum_{z \in A} \lambda( \{ z \} )$.
For $z>x$ define $\Upsilon^x(z)$ by
\[ \Upsilon^x(z) =
\exp \left( - \int_{(x,z)}  \frac{\lambda^c(dw)}{\gamma((-\infty, w)) - \Xi(w)} \right)
\prod_{u \in [x,z)} \left( 1 - \frac{\lambda (\{u \}) }{\gamma((-\infty, u)) - \Xi(u)} \right)^+ .\]
$\Upsilon^x$ can be thought of as a survivor function of a probability measure.

For $x$ in the support of $\gamma$ we define $\pi^x$ separately on $z> x$ and $z<x$.
For $z > x$, we set
\begin{equation}
\label{eq:HPdef1}
 \pi^x([z,\infty)) =  \int_z^\infty dw \exp\left( - \int_x^w \frac{dv}{ v - g(v)} \right) |d\Upsilon^x(w)|
\end{equation}
and for $z < x$,
\begin{equation}
\label{eq:HPdef2}
\pi^x((-\infty,z]) = \int_x^\infty I_{ \{ w : g(w) \leq z \}} \frac{dw}{w - g(w)} \exp \left( - \int_{x}^{w} \frac{du}{u - g(u)} \right)
\Upsilon^x(w) .
\end{equation}

The following result is an immediate consequence of Proposition 5.4 in \cite{HobsonPedersen:02}.
\begin{prop}[Hobson and Pedersen~\cite{HobsonPedersen:02}]
With $\pi^x$ as above and $\pi(dx,dy) = \gamma(dx) \pi^x(dy)$ we have $\pi \in \sM(\gamma,\lambda)$.
\end{prop}

\begin{eg}{\rm
Suppose $\gamma \sim \frac{1}{t}U[-t,t]$ and $\lambda \sim \frac{1}{2t} \delta_{-t} + \frac{1}{2t} \delta_{t}$ (as
is the case for Example~\ref{eg:uniform}). Then $Q(x) = \frac{1}{4t^2} (t^2 - x^2) I_{ \{ -t < x < t \} }$.

We find $\xi(z) = \frac{1}{4t^2} (t-z)$ for $-z < t \leq z$ (and $\xi(z)=0$ otherwise), and $g(z) = - t$ for
$-z < t \leq z$ (and $g(z)=z$ otherwise). Then $\Xi(z) = \frac{1}{4t^2} (t+z)^+$ for $z \leq t$ and
for $x \in (-t,t)$, $\Upsilon^x(z) = I_{ \{ z < t \} }$. Note that for $z \leq t$,
\( \int_x^z \frac{dw}{w-g(w)} = \int_x^z \frac{dw}{t+w} = \ln \frac{t+z}{t+x} \)
and hence $\exp (- \int_x^z \frac{dw}{w-g(w)} )= \frac{t+x}{t+z}$. Then for $-t<x<t$ and $z>x$ we have
\( \pi^x([z,\infty)) = \frac{t+x}{2t} I_{ \{ z \leq t \} }. \)

For $-t < x < t$ and $z< x$ we find $\pi^x((-\infty,z]) = I_{ \{ -t \leq z \} }\int_x^t \frac{dw}{t+w} \frac{t+x}{t+w} =
I_{ \{ -t \leq z \} } [-\frac{t+x}{t+w}]^t_{w=x} = I_{ \{ - t \leq z \} } \frac{t-x}{2t}$.

The situation can be summarised as $\pi^x = \frac{t-x}{2t} \delta_{-t} + \frac{t+x}{2t} \delta_{t}$.
}\end{eg}

\section{Extension to $[0,T]$}
\label{sec:extension}

In many of our examples the bound $K(t)$ diverges at $t=0$ and the methods of the previous section are only sufficient to
construct a process on $[\epsilon, T]$ for $\epsilon>0$. (A rare exception is Example~\ref{eg:mass0}.)
In the case where $K(t)$ diverges at zero we give two approaches to constructing a process on $[0,T]$. The first approach is based on
the Kolmogorov consistency condition and is non-constructive. For related ideas in a similar context, see Oleszkiewicz~\cite{Oleszkiewicz:08}.
The second approach requires extra assumptions, and makes use of time-reversal.

\subsection{Consruction via the Daniell-Kolmogorov Consistency Theorem}

Let $\epsilon$ be a parameter in $(0,T)$. Let $X^{(\epsilon)} = (X^{(\epsilon)}_t)_{\epsilon \leq t \leq T}$
be the Markov martingale with initial law $\mu_\epsilon$ and marginals $(\mu_t)_{\epsilon \leq t \leq T}$
constructed as in the previous section from the family of martingale transports
$(\pi^{(\epsilon)}_t)_{\epsilon < t \leq T}$ and the realisation
$\sN^{(\epsilon)} = (\sN^{(\epsilon)}_t)_{\epsilon < t \leq T}$ of the Poisson point process $(N^{(\epsilon)}_t)_{\epsilon < t \leq T}$.
Write $X^{(\epsilon)} = (X^{(\epsilon)}_t)_{\epsilon \leq t \leq T} = X^{(\epsilon)}(X^{(\epsilon)}_{\epsilon},
(\sN^{(\epsilon)}_t)_{\epsilon < t \leq T})$ to indicate that given the family of martingale transports the path
$(X^{(\epsilon)}_t)_{\epsilon \leq t \leq T}$ is a deterministic function of the initial value (at time $\epsilon$)
and the realisation of the Poisson point process.

Suppose now that we are given a family of martingale transports $(\pi_t)_{0 < t \leq T}$ and a Poisson point process
$N$ on $(0,T] \times (0,\infty) \times(0,1)$. For any $\delta$ in $(0,T)$ we can consider the restrictions of
$(\pi_t)_{0 < t \leq T}$ and $(N_t)_{0 < t \leq T}$ to $(\delta, T]$.
Write $\pi^{(\delta)}=(\pi^{(\delta)}_t)_{\delta < t \leq T} := (\pi_t)_{\delta < t \leq T}$ and
$N^{(\delta)}=(N^{(\delta)}_t)_{\delta < t \leq T} := (N_t)_{\delta < t \leq T}$ for these restrictions.
Then define
$X^{(\delta)} = (X^{(\delta)}_t)_{\delta \leq t \leq T} = X^{(\delta)}(X^{(\delta)}_{\delta},
(\sN^{(\delta)}_t)_{\delta < t \leq T})$ where $X^{(\delta)}_{\delta}$ is a random variable with law $\mu_\delta$.
It follows that
if $\delta < \epsilon$ and $X^{(\epsilon)}_{\epsilon}$ is taken to equal $X^{(\delta)}_{\epsilon}$ then
there is an identification or coupling between processes started at different times and
$(X^{(\epsilon)}_{t})_{\epsilon \leq t \leq T}$ and $(X^{(\delta)}_{t})_{\epsilon \leq t \leq T}$ are identical, omega by omega.

For $0<t_1 < t_2 < \cdots < t_n \leq T$ define
\begin{equation}
\label{eq:fdds}
\mu_{ t_1 , t_2 , \ldots , t_n} =
\sL( X^{(\epsilon)}_{t_1}, X^{(\epsilon)}_{t_2}, \ldots , X^{(\epsilon)}_{t_n}).
\end{equation}
By the above comments,
$\mu_{ t_1 , t_2 , \ldots , t_n}$ does not depend on $\epsilon$ (provided $\epsilon \leq t_1$). Also it is clear that the measures
$(\mu_{ t_1 , t_2 , \ldots , t_n})_{n \in \bbN, 0<t_1 < t_2 < \cdots < t_n \leq T}$ are consistent, since any pair
are the finite dimensional distributions of a stochastic process $(X^{(\epsilon)}_t)_{\epsilon \leq t \leq T}$
for some $\epsilon$ sufficiently small.
Hence, by the Kolmogorov consistency theorem there exists a stochastic process $\hat{X} = (\hat{X}_t)_{0<t \leq T}$
such that the finite dimensional distributions of $\hat{X}$ are given by the measures in (\ref{eq:fdds}).

Let $\bbD_{k,T} = \{ j 2^{-k} T ; 1 \leq j \leq 2^k \}$ and let $\bbD_T = \cup_{k \geq 1} \bbD_{k,T}$
so that $\bbD_T$ is the set of dyadic rationals on $(0,1]$ rescaled to $(0,T]$.
Consider $(\hat{X}_t)_{t \in \bbD_T}$. Fix $0 < t_0 < t_1 \leq T$ with $t_i \in \bbD_T$. Then with $\epsilon<t_0$,
and using the right-continuity of $X^{(\epsilon)}$,
\begin{eqnarray*}
\Prob( \hat{X}_t = \hat{X}_{t_0}, \forall t \in \bbD_T \cap [t_0,t_1])
& = & \lim_k \Prob( \hat{X}_t = \hat{X}_{t_0}, \forall t \in \bbD_{k,T} \cap [t_0,t_1]) \\
& = & \lim_k \Prob( {X}^{(\epsilon)}_t = {X}^{(\epsilon)}_{t_0}, \forall t \in \bbD_{k,T} \cap [t_0,t_1]) \\
& = & \Prob( {X}^{(\epsilon)}_t = {X}^{(\epsilon)}_{t_0}, \forall t \in [t_0,t_1]) \\
& =  & \E \left[ \exp \left( - \int_{t_0}^{t_1} R_t(X^{(\epsilon)}_{t_0}) dt \right) \right] \geq \exp( - \overline{K}(t_1 - t_0)),
\end{eqnarray*}
for some constant $\overline{K}$. In particular, on a set $\Omega_1$ of probability 1, $\hat{X} : \bbD_T \mapsto \R$
is piecewise constant and has finitely many jumps on any closed sub-interval of $(0,T]$.

On $\Omega_1$ let $X : (0,T] \mapsto \R$ be defined by $X_t =\lim_{s \downarrow t, s \in \bbD_T} \hat{X}_s$, and
on the null set $\Omega_1^c$ let $X_t = \overline{\mu}$. Then every path of $X$ is piecewise constant and right-continuous on $(0,T]$
and has finitely many jumps on
any closed sub-interval of $(0,T]$.

Suppose either that $K(t)$ is bounded on $(0,T)$ or that $\mu_t$
converges in $L^1$ to a point mass, necessarily at $\overline{\mu}$. In the former case let $X_0 =\lim_{s \downarrow 0, s \in \bbD_T} \hat{X}_s$
(which exists since $\hat{X}$ has only finitely many jumps in $(0,T]$ if $K$ is bounded) and in the latter case, let $X_0 =\overline{\mu}$.

It remains to check that $X$ has the desired marginals, but for $t>0$ this is immediate given Assumption~\ref{ass:2}(e):
\[ \Prob(X_t \in B) = \lim_{s \downarrow t, s \in \bbD_T} \Prob ( \hat{X}_s \in B )
= \lim_{s \downarrow t, s \in \bbD_T} \mu_s(B) = \mu_t(B). \]
When $t=0$, if $K(t)$ is bounded on $(0,T)$ then the same argument applies, otherwise $X_0 \sim \mu_0 \sim \delta_{\overline{\mu}}$ by definition.

The fact that $X$ is a martingale (relative to its own natural filtration) follows as in
Oleszkiewicz~\cite{Oleszkiewicz:08}.

\subsection{Extension to $[0,T]$ via time reversal}
Under some extra conditions it is possible to sketch an alternative and more direct construction of a process on $[0,T]$.
In particular in this section, in addition to the hypotheses of Assumptions~\ref{ass:1}, \ref{ass:2} and \ref{ass:3} we assume that we
the family of marginals $(\mu_t)_{t \in [0,T]}$ is such that $\lim_{t \downarrow 0} U(t,x)
= | x - \overline{\mu}|$ (and then $\mu_0 \sim \delta_{\overline{\mu}}$) and further that the Radon-Nykodym
derivative $\lambda_t(dy)/\mu_t(dy)$ is bounded in $y$ by a
function $\tilde{K}(t)$ which is bounded on every interval of the form $[\epsilon,T]$ with $\epsilon>0$.

Throughout this section we assume $T$ is fixed.

Temporarily fix $\epsilon$.
Using the ideas of the previous section we can construct $(X^{(\epsilon)}_t)_{\epsilon \leq t \leq T}$ such that $X^{(\epsilon)}_t \sim \mu_t$. Now let
$\tilde{X}^{(\epsilon)} =(\tilde{X}^{(\epsilon)}_s)_{0 \leq s \leq T-\epsilon}$ be given by
$\tilde{X}^{(\epsilon)}_s = X^{(\epsilon)}_{T-s}$. Then $\tilde{X}^{(\epsilon)}$ is the time-reversal of $X^{(\epsilon)}$ and $\tilde{X}^{(\epsilon)}_s
= X^{(\epsilon)}_{T-s} \sim \mu_{T-s} := \tilde{\mu}_s$
Indeed,
the process $\tilde{X}$ is a pure-jump Markov process with state-dependent jump intensities $\tilde{R}_s(y)$
and conditional jump measures $\tilde{\pi}^y_s$ where
\[ \tilde{R}_s(y) := \frac{\lambda_{T-s}(dy)}{\mu_{T-s}(dy)} \hspace{20mm}
\tilde{\pi}^y_s(dx) := \frac{\pi_{T-s}(dx,dy)}{\lambda_{T-s}(dy)} = \frac{\gamma_{T-s}(dx) \pi^x_{T-s}(dy)}{\lambda_{T-s}(dy)} . \]
In particular, $\tilde{X}^{(\epsilon)}$ does not depend on $\epsilon$ in the sense that if $\tilde{X}^{(\epsilon)}_0 = \tilde{X}^{(\delta)}_0$ then
$\tilde{X}^{(\epsilon)}$ and $\tilde{X}^{(\delta)}$ are pathwise identical on the set $[0,(T-\delta)
\wedge (T-\epsilon)]$ on which they are both defined.

Hence we may drop the superscript. Let $\tilde{X}$ be the pure-jump Markov process with initial law $\tilde{X}_0 \sim  \tilde{\mu}_0 =\mu_T$
defined using $\tilde{R}_s$ and $\tilde{\pi}^y_s$ as in the Section~\ref{ssec:construction}, except that here we take $\tilde{X}$ to be left-continuous.
(Note that $\tilde{\pi}^y_{s}$ does not respect the martingale property, so that the process $\tilde{X}$ will not be a martingale.)
Then $\tilde{X}$ is defined on $[0,T)$, and $\tilde{X}_s \sim \tilde{\mu}_{s} = \mu_{T-s}$. Now set $\overline{X}_0 \sim \delta_{\overline{\mu}}$ and
$\overline{X}_t = \tilde{X}_{T-t}$. Then $\overline{X} = (\overline{X}_t)_{0 \leq t \leq T}$ is well-defined on $[0,T]$ and
$\overline{X}_t \sim \mu_t$. Further $\overline{X}$ agrees with $X^{(\epsilon)}$ on $[\epsilon, T]$ and so $\overline{X}$ inherits the martingale property from
$X^{(\epsilon)}$.

\begin{eg} {\rm
Suppose $\mu_t \sim e^{-t} \delta_0 + (1-e^{-t}) U[-1,1]$. Then $\gamma_t \sim e^{-t} \delta_t$ and $\lambda_t \sim e^{-t}U[-1,1]$.
We find that for $y \in (-1,1) \setminus \{0 \}$, $\tilde{R}_s(y) = \frac{e^{-(T-s)}}{1-e^{-(T-s)}}$ and
$\tilde{\pi}^y_s = \delta_0$.

Then the time reversed process is well-defined on $[0,T]$. The process starts at 0 with probability $e^{-T}$ and if it starts at 0 then it remains
constant at zero throughout.
Otherwise it starts at a non-zero point chosen according to the uniform distribution on $[-1,0)\cup (0,1]$, and stays constant until
the time of the first event of a time-inhomogeneous Poisson process with rate $\frac{e^{-(T-s)}}{1-e^{-(T-s)}}$,
at which point it jumps to zero. Zero acts as an absorbing state.

In forward time the process starts at 0 and jumps to a uniformly distributed location on $[-1,1]\setminus \{0 \}$ at rate 1.
}\end{eg}

\begin{eg}{\rm
Suppose that for $t \in [0,T]$, $\mu_t \sim U[-t,t]$. In this case $\lambda_{T-t}$ is not absolutely equivalent with respect to $\mu_{T-t}$, and all calculations
for this example should be viewed as formal calculations. (Nonetheless, this example is sufficiently concrete that it is easy to verify that the proposed processes
have all the right properties.) We find $\tilde{R}_s(y) = 0$ except at $y = \pm (T-s)$, and at those points it is infinite. The interpretation is that $\tilde{X}$
only jumps when $\tilde{X}_s = T-s$. Conditional on their being a jump we find
$\tilde{\pi}^{T-s}_s(dx) = \frac{(T-s+x)}{2(T-s)^2} dx$  and $\tilde{\pi}^{-(T-s)}_s(dx) = \frac{(T-s-x)}{2(T-s)^2} dx$.

The process $\tilde{X}$ can be described as follows. The initial value $x_0$ is uniform on $[-T,T]$. The process remains constant until $\tau_1 = T- |x_0|$.
At time $\tau_1$ the process jumps to a new point $x_1$ in the interval $(-(T-\tau_1), T- \tau_1)$ chosen according to the density
$f(x_1) = \frac{(T - |x_0|+x_1 sgn(x_0))}{2(T-|x_0|)^2}$.
The process can then be constructed inductively. Suppose that after $k$ jumps and at time $s$ the process is at $x_k$
where necessarily $x_k \in [-(T-s),T-s]$. The process remains constant until $\tau_{k+1} = T - |x_{k}|$ at which point the process jumps to a new point in
$(-|x_k|, |x_k|) = (-(T-\tau_{k+1}),T-\tau_{k+1})$ chosen with density $\frac{(T - |x_k|+x_{k+1} sgn(x_k))}{2(T-|x_k|)^2}$.
}\end{eg}

\section{Minimising the expected total variation}
\label{sec:minV}

In this section we want to explain how the ideas of Section~\ref{sec:candidate}
can be used to construct an extremal process consistent with a given set of marginals under an additional
assumption on the relationship between the measures $\gamma_t$ and $\lambda_t$. In particular, the process
we construct will have smallest expected total variation amongst all processes consistent with a given set
of marginals. As an application we construct the fake Brownian motion with smallest expected total variation.

Our analysis is motivated by the martingale coupling constructed in Hobson and Klimmek~\cite{HobsonKlimmek:15}. A key assumption in the construction
of that paper is that mass is moved using a martingale transport from an interval in the centre to the tails. We make a similar assumption, which is augmented
with some further regularity conditions. A similar condition to Assumption~\ref{ass:dispersion}(a) is stated in Henri-Labord\`{e}re
and Touzi~\cite{Henri-LabordereTouzi:15} as the difference in the cumulative distribution functions (of $\lambda_t$ and $\gamma_t$)
having a unique local maximum.

\begin{ass} 
\label{ass:dispersion}
\begin{enumerate}
\item[(a)] (Dispersion assumption)
There exists a family of open intervals $(E_t)_{t \in \bbT}$ such that $\supp(\gamma_t) \subseteq E_t$ and $\supp(\lambda_t) \subseteq E_t^c$.
Moreover, the intervals $(E_t)_{t \in \bbT}$ are increasing: for $s<t$ we have $E_s \subseteq E_t$;
\item[(b)] (Regularity)
$U(t,x) \in C^{1,2}$ and $Q_t(x)$, $Q_t'(x)$ and $\Gamma_t'(x)$ are continuous in $t$ for all $x$.
\item[(c)] (strict positivity of densities on their domains) For each $t \in \bbT$
$\dot{\rho}(t,x) < 0$ for $x \in E_t$ (and then $E_t =  (\ell_{\gamma}(t), r_{\gamma}(t))$) and $\dot{\rho}(t,x) > 0$ for $x \in
({\ell}_{\lambda}(t),\ell_{\gamma}(t)) \cup (r_{\gamma}(t),r_{\lambda}(t))$;
\item[(d)] (Finite activity) $\int (|x| + 1) |\dot{\rho}(t,x)| dx < \infty$ for each $t \in \bbT$
and $\int_{t \in \bbT} dt \int (|x| + 1) |\dot{\rho}(t,x)| dx < \infty$.
\end{enumerate}
\end{ass}

Consider the Brownian motion (Example~\ref{eg:bm}) and exponential Brownian motion (Example~\ref{eg:ebm}) examples.
For the former $E(t) = (-\sqrt{t},\sqrt{t})$ and $I_t = \R$. For the latter, $E(t) = (\exp( -\sqrt{t+t^2/4}), \exp( \sqrt{t + t^2/4}))$ and $I_t = (0,\infty)$.
In both cases all the elements of Assumption~\ref{ass:dispersion} are satisfied.
However, in the uniform example (Example~\ref{eg:uniform}) and in Example~\ref{eg:mass0}, we find $U(t,x) \notin C^{1,2}$, and these examples do not satisfy
the regularity conditions (b) and (c) of Assumption~\ref{ass:dispersion}.
Nonetheless, the dispersion assumption (Assumption~\ref{ass:dispersion}(a)) is satisfied in Example~\ref{eg:uniform}: mass is moving from a central region
to the tails and there exists an open interval $E_t = (\ell_\gamma(t),r_\gamma(t))=(-t,t)$ such that
$\dot{\rho}(t,x) < 0$ on $E_t$, and $\supp(\lambda_t) \subseteq E_t^c$.
In Example~\ref{eg:mass0} the spirit of Assumption~\ref{ass:dispersion}(a) is satisfied in the sense that mass is moving from a central region to the tails, but in this
case $E_t = \{0 \}$ is not open.
For these examples it is possible to prove optimality directly, see the discussion in
Example~\ref{eg:uniform2} below.

\subsection{Construction of the martingale transport for fixed $t$}
\label{ssec:constructionfixedt}
In this section we will consider $t$ to be fixed and generally write $t$ as a subscript. Later we will consider $t$ as a parameter.

Our goal is to show how to write down a martingale transport $\pi_t \in \sM(\gamma_t,\lambda_t)$ under Assumption~\ref{ass:dispersion}.
In particular, the idea is that $\pi_t$ has a decomposition $\pi_t(dx,dy) = \gamma_t(dx) \pi^x_t(dy)$ where $\pi^x_t$ has a binomial distribution:
\begin{equation}
\label{eq:pidef}
 \pi^x_t = \frac{b_t(x) - x}{b_t(x)-a_t(x)} \delta_{a_t(x)} + \frac{x - a_t(x)}{b_t(x)-a_t(x)} \delta_{b_t(x)} .
\end{equation}
Thus, for this candidate martingale transport, mass at $x$ moves down to $a_t(x)$ or up to $b_t(x)$, where $a_t(x) < x < b_t(x)$,
and the probabilities of up and down moves are such that the
martingale property is satisfied. Our goal is to explain how to choose $a_t(x)$ and $b_t(x)$ such that  $\pi_t \in \sM(\gamma_t,\lambda_t)$.
A key element of the construction will be that $a_t(x)$ and $b_t(x)$ are decreasing functions of $x$.

Following Hobson and Klimmek~\cite{HobsonKlimmek:15} for $x \in E_t$ and $y>x$ define
$\sQ_{t,x}(y) = Q_t(x) + Q_t'(x)(y-x) - Q_t(y)$.
For $x \in E_t$ define
\begin{equation}
\label{eq:phidef}
 \phi_t(x) = {\max}_{\alpha < x < \beta} \left\{
\frac{\sQ_{t,x}(\beta) - Q_t(\alpha)}{\beta - \alpha} \right\}
\end{equation}
and let the maximum be attained at the pair
$(\alpha = a_t(x),\beta=b_t(x))$.
Under Assumption~\ref{ass:dispersion} (and especially the dispersion and strict positivity elements)
$Q_t$ is strictly convex and increasing on $(\ell_{\lambda}(t), \ell_\gamma(t))$,
strictly concave on $(\ell_\gamma(t),\ell\gamma(t))$
and strictly convex and decreasing on $(r_\gamma(t), r_\lambda(t))$.
This guarantees the existence and uniqueness of optimisers $a_t(x)$ and $b_t(x)$ in the right-hand-side of (\ref{eq:phidef})
with $a_t(x) < \ell_\gamma(t)$ and $b_t(x) > r_\gamma(t)$.

\begin{prop}
\label{prop:differentiable}
$a_t(\cdot)$ is continuously differentiable in $x$ on $E_t$ and satisfies
\[ a'_t(x) \dot{\rho}(t, a_t(x)) = \frac{b_t(x)-x}{b_t(x)-a_t(x)}
\dot{\rho}(t, x) . \]
Similarly $b_t(\cdot)$ is continuously differentiable on $E_t$ and satisfies
\[ b'_t(x) \dot{\rho}(t, b_t(x)) =
- \frac{x-a_t(x)}{b_t(x)-a_t(x)}
\dot{\rho}(t,x). \]
\end{prop}

\begin{proof}
We prove the result for $a_t$, the proof for $b_t$ being similar.

The construction of $a_t$, $b_t$ given in Hobson and Klimmek~\cite{HobsonKlimmek:15} is a martingale transport of probability measures, but generalises easily to
pairs of measures with the same total mass which are increasing in convex order. By construction
$a_t(x) = (Q_t')^{-1} \circ \phi_t \circ {\gamma}_t((-\infty,x]) $. By the continuity and strict positivity of $\dot{\rho}(t,x)$ on $(\ell_\lambda(t),
\ell_{\gamma}(t))$, $(Q'_t)^{-1}$ is continuously differentiable, as is $\gamma_t((-\infty,x])$. Further, by Lemma 5.1 and the argument after Lemma 5.2
of Hobson and Klimmek~\cite{HobsonKlimmek:15}, $\phi_t$ is
continuously differentiable. Hence $a_t(\cdot)$ is continuously differentiable. The expression for the derivative is taken from Section 3.1 of
\cite{HobsonKlimmek:15}. Note that \cite{HobsonKlimmek:15} covers the case of general probability measures
(satisfying Assumption~\ref{ass:dispersion}(a)) including measures with atoms.
\end{proof}

\begin{cor} Define $\pi_t$ by $\pi_t(dx,dy) = \gamma_t(dx) \pi^x_t(dy)$ where $\pi_t^x$ is as given in {\rm (\ref{eq:pidef})}. Then $\pi_t \in \sM(\gamma_t,\lambda_t)$ and
$\pi_t^x(B)$ is measurable in $x$ for all measurable sets $B \subseteq I$.
\end{cor}
\begin{proof}
Suppose $H$ is a bounded, continuous test function. Then, using Proposition~\ref{prop:differentiable} in the second line, and a change of variables in the third,
\begin{eqnarray*}
\int_x \int_y \gamma_t(dx) \pi_t^x(dy) H(y)
& = & \int_{E_t} \dot{\rho}(t,x)^- dx \left( \frac{(b_t(x)-x)}{b_t(x) - a_t(x)} H(a_t(x)) + \frac{(x - a_t(x))}{b_t(x) - a_t(x)} H(b_t(x))
                 \right) \\
& = & - \int_{E_t} H(a_t(x)) a'_t(x) \dot{\rho}(t,a_t(x)) dx - \int_{E_t} H(b_t(x)) b_t'(x) \dot{\rho}(t,b_t(x)) dx  \\
& = & \int_{(\ell_\lambda(t),\ell_\gamma(t))} H(y) \dot{\rho}(t,y) dy
+ \int_{(r_\gamma(t),r_\lambda(t))} H(y) \dot{\rho}(t,y) dy \\
& = & \int_{I_t} \lambda_t(dy) H(y)
\end{eqnarray*}
and it follows that $\pi_t \in \sM(\gamma_t,\lambda_t)$. Measurability in $x$ of $\pi^x_t(B)$ is clear from the monotonicity of $a_t$ and $b_t$.
\end{proof}


Now we introduce some auxiliary variables which will play a role in the evaluation of the minimal expected total variation.
Fix $x_0 \in \cap_{t \in \bbT} E_t$. For $x \in \overline{E}_t = [ \ell_{\gamma}(t), r_{\gamma}(t) ]$ define
\[ \theta_t(x) = \int_{x_0}^x \frac{2}{b_t(z)-a_t(z)} dz \hspace{10mm} \psi_t(x) = \int_{x_0}^x \frac{2(x-x_0) - b_t(z) - a_t(z)}{b_t(z)-a_t(z)} dz. \]
Extend these definitions to $x \notin \overline{E}_t$ by setting
\[ \theta_t(x) = \left\{ \! \begin{array}{ll} \theta_t( a^{-1}_t(x)) & x < \ell_\gamma(t) \\
                                           \theta_t( b^{-1}_t(x)) & x > r_\gamma(t) \end{array} \! \right\}
    \hspace{3mm} \psi_t(x) =  \left\{  \! \begin{array}{ll} \psi_t( a^{-1}_t(x)) + (a_t^{-1}(x)-x)(1 - \theta_t(a_t^{-1}(x))) & x < \ell_\gamma(t) \\
                                           \psi_t( b^{-1}_t(x)) + (x-b_t^{-1}(x))(1 + \theta_t(b_t^{-1}(x))) & x >r_\gamma(t)
                                           \end{array} \! \right\}
\]
Set $L_t(x,y) = L_{\psi_t,\theta_t}(x,y) = |y-x| + \psi_t(x) + \theta_t(x)(y-x) - \psi_t(y)$.
\begin{prop}[Hobson and Klimmek~\cite{HobsonKlimmek:15}]
$L_t(x,y) \geq 0$, with equality for $y \in \{ a_t(x), x, b_t(x) \}$. Further, $\psi_t(x) \leq |x-x_0|$ and $\psi_t$ is convex on $E_t$ and concave on
$(\ell_\lambda(t), \ell_{\gamma}(t))$ and $(r_{\gamma}(t),r_\lambda(t))$.
\label{prop:hk}
\end{prop}
\begin{proof}
The only part of this result which is not contained in Theorem 4.5 of Hobson and Klimmek~\cite{HobsonKlimmek:15} is the fact that $\psi_t(x) \leq |x-x_0|$. But
\[ 0 \leq L_t(x_0,x) = |x-x_0| + \psi_t(x_0) + \theta_t(x_0)(x-x_0) - \psi_t(x) = |x-x_0| - \psi_t(x). \]
\end{proof}

\begin{cor}
\label{cor:hk}
For any random variables $X, Y \in L^1$ with $\E[(Y-X)|X] = 0$ we have $\E[\psi_t(Y) - \psi_t(X) ] \leq \E[|Y-X|]$.
\end{cor}

\subsection{Continuity in $t$}
Now we consider the dependence of the various quantities on $t$. Where appropriate we switch to writing $t$ as an argument rather than as a subscript;
for example we now write $a(t,x)$ in place of $a_t(x)$.

\begin{lem}
$a(t,x)$ and $b(t,x)$ are continuous in $t$. Further, $\psi(t,x)$ is continuous in $t$.
\end{lem}

\begin{proof}
We have the representation $a(t,x) = (Q'_t)^{-1} \circ \phi_t \circ \gamma_t (( -\infty,x])$. By assumption, $Q_t'$ and $\Gamma_t'$ are continuous in $t$.
Since $\phi(t,x)$ is continuous in $x$, continuity of $a$ will follow if $\phi(t,x)$ is continuous in $t$.

Recall (\ref{eq:phidef}) and the fact that $\phi(t,x)$ is the slope from $(a(t,x),Q(t,a(t,x)))$ to $(b(t,x),\sQ_{t,x}(b(t,x)))$.
Then, using continuity of $Q_t$ and $Q_t'$ in $t$ (Assumption~\ref{ass:dispersion}(b))
\[ \phi(t',x) \! = \! \sup_{\alpha<x<\beta} \! \frac{\sQ_{t',x}(\beta) - Q(t',\alpha)}{\beta - \alpha} \geq \frac{\sQ_{t',x}(b(t,x)) - Q(t',a(t,x))}{b(t,x) - a(t,x)}
\stackrel{t' \rightarrow t}{\longrightarrow}
\frac{\sQ_{t,x}(b(t,x)) - Q(t,a(t,x))}{b(t,x) - a(t,x)} =  \! \phi(t,x). \]
Conversely, given $\eta >0$, for $t'$ sufficiently close to $t$ we have $Q(t,y) > Q(t',y) - \eta$ and $\sQ_{t,x}(y) < \sQ_{t,x}(y)+ \eta(1 + (y-x))$,
and then if $r_\gamma(t') - \ell_\gamma(t') > r_{\gamma}(t-) - \ell_{\gamma}(t-) - \eta$,
\[ \phi(t',x) \!=\! \sup_{\alpha<x<\beta} \! \frac{\sQ_{t',x}(\beta) - Q(t',\alpha)}{\beta - \alpha} <
\! \sup_{\alpha<x<\beta} \! \frac{\sQ_{t,x}(\beta) - Q(t,\alpha)}{\beta - \alpha} + \eta +  \frac{2 \eta}{\beta-\alpha}
< \phi(t,x) + \eta + \frac{2 \eta}{r_{\gamma}(t-) - \ell_{\gamma}(t-) - \eta}  . \]
Hence for fixed $x$, $\phi(t,x)$ is continuous in $t$.

Continuity of $b(t,x)$ in $t$ is similar, from which continuity of $\psi(t,x)$ follows.
\end{proof}

\begin{cor}
Suppose Assumptions~\ref{ass:1}, \ref{ass:2} and \ref{ass:dispersion} hold. Then $X_t \sim \mu_t$ for $t \in \cbbT$.
\end{cor}

\begin{proof}
It is clear that under Assumption~\ref{ass:dispersion} $R_t(x)$ and $\pi^x_t$ are such that the measurability requirements of
Assumption~\ref{ass:2}(d) and Assumption~\ref{ass:3} are
satisfied and hence by Theorem~\ref{thm:mimic} that $X_t \sim \mu_t$.
\end{proof}

\subsection{Minimising expected total variation: the primal approach}
The idea now is to show that, under Assumption~\ref{ass:dispersion}, the process we have constructed has the
smallest possible expected total-variation in the class of martingales with the specified marginals.

Suppose $\epsilon \geq 0$ and let $\cbbT = [\epsilon,T]$.
Let $P^{\cbbT}=\{ t_0,
t_1,
\ldots
t_{N} \}$ be a partition of $\cbbT$ of dimension $N$, i.e. a vector
of length $N+1$ with
$t_0=\epsilon$, $t_k > t_{k-1}$ and $t_{N}=T$. We say $\sP^{\cbbT}
= (P^{\cbbT}_n)_{n \geq 1}$
is a dense
sequence of nested partitions of $[\epsilon,T]$, if $ P^{\cbbT}_n = \{ t^n_0,
t^n_1, \ldots t^n_{N(n)} \}$ is a partition for each $n$, $P^{\cbbT}_{n}
\subseteq \{ t_0^{n+1}, t^{n+1}_1, \ldots
t^{n+1}_1, \ldots t^{n+1}_{N(n+1)} \} = P^{\cbbT}_{n+1}$ and $\lim_{n \uparrow \infty} \max_{1 \leq k \leq
N(n)} |t^n_k - t^n_{k-1}| = 0$.

Let $P_{U,n}^{\cbbT}=\{ t^{U,n}_0,
t^{U,n}_1, \ldots t^{U,n}_{2^n} \}$ be the uniform partition of $[\epsilon,T]$ of dimension $2^n$
in which $t^{U,n}_k = \epsilon + (T-\epsilon) k 2^{-n}$, and let
$\sP^{\cbbT}_U
= (P^{\cbbT}_{U, n})_{n \geq 1}$. Then $\sP^{\cbbT}_U$ is a
dense sequence of nested partitions.

\begin{defn} Let $f : [\epsilon, T] \mapsto \R$ be any function.
The total variation of $f$ along a partition $P^{\cbbT}$ is
\[ \sV(P^{\cbbT}, f) =
\sum_{k=1}^{N} | f(t_k) - f(t_{k-1})| .
\]

The total variation of $f$ along a dense sequence of nested partitions $\sP^{\cbbT}$ is
\[ \sV(\sP^{\cbbT}, f) =
\lim_n \sum_{k=1}^{N(n)} | f(t^n_k) - f(t^n_{k-1})| .
\]
The fact that the sequence is nested means that the limit exists.

The total variation of $f$ is $\sV^{\cbbT}(f) = \sup \sV(\sP^{\cbbT}, f)$ where the supremum is taken over dense sequences of nested
partitions. Note that $\sV^{\cbbT}(f) \geq \sV(\sP^{\cbbT}_U,f)$.
\end{defn}

\begin{defn}
Let $Z = (Z_t)_{t \in \cbbT}$ be a stochastic process. Define the expected
total variation $V^{\cbbT}(Z)$ of $Z$ by
$V^{\cbbT}(Z) = \E[\sV^{\cbbT}(Z)] = \E[\sup \sV(P^{\cbbT},Z)]$, where the supremum is taken over partitions which may
depend upon the realisation of $Z$.

For the sequence of uniform partitions $\sP^{\cbbT}_U$ define the expected total variation along the uniform
partition $V(\sP^{\cbbT}_U,Z)$ by
$V(\sP^{\cbbT}_U,Z) = \E[ \lim_n \sV( P^{\cbbT}_{U,n}, Z)]$.
By the final remark in the
previous definition, $V^{\cbbT}(Z) \geq V(\sP^{\bbT}_U,Z)$.
\end{defn}


Suppose $\epsilon > 0$ and $\bbT = (\epsilon,T)$.

\begin{thm}
\label{thm:bound}
Suppose Assumptions~\ref{ass:1}, \ref{ass:2} and \ref{ass:dispersion} hold.
Let $Y$ be any martingale with marginals $(\mu_t)_{t \in \cbbT}$. Then,
\begin{equation}
\label{eq:inequality}
V^{\cbbT}(Y) \geq V( \sP_U^{\cbbT}; Y) \geq \int_{\bbT} dt \int dx \psi_t(x) \dot{\rho}(t,x)  = \int_{\bbT} dt \int \psi_t(x) q_t(dx).
\end{equation}
\end{thm}

\begin{proof} We only need to prove the second inequality.
Let $[t]_n = \max \{ t_k \in P^{\cbbT}_{U,n} ; t_k \leq t \}$. Then $[t]_n \uparrow t$ and
by Corollary~\ref{cor:hk}, for $s<t$ we have,
$\E[|Y_t - Y_s|] \geq \E[ \psi(s,Y_t) - \psi(s, Y_s) ]$,
\begin{eqnarray*}
V(\sP^{\cbbT}_{U,n},Y) & = & \E \left[ \sum_{k=1}^{2^n} |Y_{t^{U,n}_k} - Y_{t^{U,n}_{k-1}}|   \right] \\
& \geq & \sum_{k=1}^{2^n} \E \left[ \left( \psi(t^{U,n}_{k-1} , Y_{t^{U,n}_k}) - \psi({t^{U,n}_{k-1}} , Y_{t^{U,n}_{k-1}}) \right)  \right] \\
& = & \sum_{k=1}^{2^n} \int_{I}  dx \; \psi(t^{U,n}_{k-1}, x) \left[ \rho(t^{U,n}_{k},x) - \rho(t^{U,n}_{k-1},x) \right] \\
& = &  \int dx \int_{\bbT} \psi([t]_n,x) dt
\dot{\rho}(t,x) dt \\
& \longrightarrow & \int dx \int_\epsilon^T dt \psi(t,x) \dot{\rho}(t,x) dt
\end{eqnarray*}
where we use the fact that $\psi([t]_{n},x) \rightarrow \psi(t,x)$, $\psi(t,x) \leq |x-x_0|$,
$\int_{\bbT} \int (|x| + 1) |\dot{\rho}(t,x)| dx < \infty$ and dominated convergence.

\end{proof}

It remains to show that this bound is best possible, or equivalently for the process constructed
in this section there is equality in (\ref{eq:inequality}). This
follows from the fact that $L_t(x,y)\equiv 0$ for $y \in \{a(t,x),x,b(t,x)\}$.

\begin{thm}
\label{thm:attain}
Let $X$ be the process constructed in Section~\ref{sec:candidate} using the family of martingale transports given in
(\ref{eq:pidef}). Then
\begin{equation}
\label{eq:attain1}
 V^{\cbbT}(X) = \int_{\bbT} dt \int dx \psi(t,x)
\dot{\rho}(t,x).
\end{equation}
Hence in the class of martingales with marginals $(\mu_t)_{t \in \cbbT}$, $X$ minimises the expected total variation.
\end{thm}

\begin{proof}
Firstly note that from easy properties about zero mean random variables
taking only two values
\begin{eqnarray*}
 \E\left[ \left. | X_{t+dt}-X_t | \right| X_t=x \right] & = &
 \mbox{ Rate of jumps}
\times dt \times \E\left[ \left. | \Delta X_t | \right| \mbox{Jump at
$t$}, \; X_t=x \right] \\
& = & R_t(x) \frac{2(b(t,x)-x)(x-a(t,x))}{b(t,x)-a(t,x)} dt
\end{eqnarray*}

Recall that by construction $X$ is right-continuous, and almost surely has only finitely many jumps
in $\cbbT = [\epsilon,T]$. We may assume that the jump times are elements of the partition and then
\begin{eqnarray*}
V^{\cbbT}(X) = \E \left[ \sum_{ \epsilon < t \leq T} | \Delta X_t| \right]
& = & \int_\epsilon^T dt \int_{E_t} \rho(t,x) R_t(x) \frac{2(b(t,x)-x)(x-a(t,x))}{b(t,x)-a(t,x)} dx  \\
& = &  \int_\epsilon^T dt \int_{E_t} 2\dot{\rho}(t,x)^- \frac{(b(t,x)-x)(x-a(t,x))}{b(t,x)-a(t,x)} dx . 
\end{eqnarray*}

The result will follow if we can show that
\begin{equation}
\label{eq:3b}
\int dx \psi_t(x) \dot{\rho}(t,x)
 =   2 \int_{E_t} \dot{\rho}(t,x)^- \frac{(b(t,x)-x)(x - a(t,x))}{b(t,x)-a(t,x)} dx . 
\end{equation}


We have
\begin{equation}
\label{eq:attain} \int dx \psi_t(x) \dot{\rho}(t,x) dx =
 -   \int_{E_t}  \psi(t,x) \dot{\rho}(t,x)^- dx
 + \int_{E^c_t}  \psi(t,y) \dot{\rho}(t,y) dx .
\end{equation}
Then, from the fact that $L_t(x,y) = 0$ for $y \in \{ a_t(x), x, b_t(x) \}$, for $y = a(t,x)$ we have
$\psi_t(y) = \psi(t,y) = \psi(t,x) + (x-y)(1- \theta(t,x))$ and for
$y = b(t,x)$ we have
$\psi_t(y) = \psi(t,y) = \psi(t,x) + (x-y)(-1- \theta(t,x))$. Then, by Proposition~\ref{prop:differentiable},
\begin{eqnarray*}
\int_{y \leq \ell_\gamma(t)} dy \psi(t,y) \dot{\rho}(t,y)
& = & - \int_{E_t} dx a'(t,x) \psi(t,a(t,x)) \dot{\rho}(t,a(t,x)) \\
& = &  \int_{E_t} dx \left[
\psi(t,x) +(x-a(t,x)) - (x -
a(t,x))\theta(t,x) \right] \frac{b(t,x)-x}{b(t,x)-a(t,x)}
\dot{\rho}(t,x)^- .
\end{eqnarray*}
Similarly,
\[ \int_{y \geq r_\gamma(t)} dy \psi(t,y) \dot{\rho}(t,y) =
 \int_{E_t} dx \left[
\psi(t,x)- (x-b(t,x)) - (x -
b(t,x))\theta(t,x) \right] \frac{x-a(t,x)}{b(t,x)-a(t,x)}
\dot{\rho}(t,x)^- .  \]
Substituting these last two expressions into (\ref{eq:attain}) we find that the
terms involving $\psi(t,x)$ and $\theta(t,x)$ cancel and (\ref{eq:3b}) follows.
\end{proof}

\begin{rem}{\rm
Let $H_t = \int \psi(t,x) \rho(t,x) dx$. Suppose that
that $\int_{t \in \bbT} dt  \int_x dx  |\dot{\psi}(t,x)| {\rho}(t,x) < \infty$
and note that Assumption~\ref{ass:dispersion}(d) implies $\int_{t \in \bbT} dt  \int_x dx  |\psi(t,x)| | \dot{\rho}(t,x)| < \infty$. Then
\[ H_T = H_\epsilon + \int_\epsilon^T \dot{H}_t dt = H_\epsilon + \int_\epsilon^T dt \int \dot{\psi}(t,x) \rho(t,x) dx
+ \int_\epsilon^T dt \int \psi(t,x) \dot{\rho}(t,x) dx ,\]
and hence if $Y= (Y_t)_{t \in \cbbT}$ is a martingale with marginals $(\mu_t)_{t \in \cbbT}$ then
\begin{equation}
\label{eq:inequality2}
 V^{\cbbT}(Y) \geq \int \psi_T(x)\rho(T,x) dx  - \int \psi_\epsilon(x) \rho(\epsilon,x) dx - \int_\epsilon^T dt \int \dot{\psi}(t,x) \rho(t,x) dx.
\end{equation}
}
\end{rem}

Now consider the case where $\epsilon=0$ and $\bbT=(0,T)$. Results in this case follow on taking the limit $\epsilon \downarrow 0$.
\begin{cor} Suppose $\int_0^T ( |x|\vee 1) \xi_t(dx) < \infty$. Let $Y= (Y_t)_{t \in [0,T]}$ be any martingale with marginals $(\mu_t)_{t \in [0,T]}$,
and let $X$ be the martingale constructed using the martingale transports in (\ref{eq:pidef}). Then
$V^{[0,T]}(Y) \geq V^{[0,T]}(X) = \int_0^T dt \int dx \: \psi(t,x) \dot{\rho}(t,x)$ and in the class of martingales
with marginals $(\mu_t)_{t \in [0,T]}$, $X$ has minimal expected total variation.

Suppose further that $\int \psi_\epsilon(x) \rho(\epsilon,x) dx \stackrel{\epsilon \downarrow 0}\rightarrow 0$ and
$\int_{0}^T dt  \int_x dx  |\dot{\psi}(t,x)|{\rho}(t,x) < \infty$. Then
\begin{equation}
\label{eq:0T}
V^{[0,T]}(Y) \geq V^{[0,T]}(X) = \int_0^T dt \int \psi(t,x) q_t(dx) =
\int \psi_T(x)\mu_T(dx)  - \int_0^T dt \int \dot{\psi}(t,x) \mu_t(dx) .
\end{equation}
\end{cor}

\subsection{The dual problem and a pathwise representation}
The inequalities (\ref{eq:inequality}) and (\ref{eq:inequality2}) are lower bounds on expected total variation. It is well known from the literature on
model-free option pricing (see, for example, Hobson~\cite{Hobson:98,Hobson:11} and Beiglb\"ock and Nutz~\cite{BeiglbockNutz:13}) that such bounds are
often related to pathwise martingale inequalities. Again the key to this result is the inequality $L_t(x,y) \geq 0$.

Note that in the following theorem we do not make any assumption that $Z$ is a martingale,
or that $Z$ has the correct marginals.

\begin{thm}
\label{thm:dual} Suppose $\bbT = (\epsilon,T)$ with $\epsilon > 0$.
Suppose $\dot{\psi}(t,x)$ and $\theta(t,x)$ are continuous and bounded on $\bbT \times I$ and suppose $L_{t}(x,y) = L_{\psi_t, \theta_t}(x,y) \geq 0$ for all
$(x,y) \in I \times I$ and all $t \in \bbT$.
Then for all paths $Z(\omega) = (Z_{u}(\omega))_{u \in \cbbT}$
\begin{equation}
\label{eq:dual} \sV^{\cbbT}(Z(\omega)) 
\geq \psi(T,Z_T(\omega)) - \psi(\epsilon,Z_\epsilon(\omega)) - \int_\epsilon^T \dot{\psi}(u,Z_u(\omega)) du
- \int_\epsilon^T \theta(u,Z_u(\omega)) dZ_u(\omega) .
\end{equation}
\end{thm}

\begin{proof}
Note that if the path is not of finite variation then the inequality is trivially satisfied.
So fix $\omega$ and suppose $Z= (Z_t(\omega))_{t \in \cbbT}$
is of finite variation.
We have $L_{\psi_t, \theta_t}(Z_t,Z_{t+h}) \geq 0$ and hence
\begin{equation}
\label{eq:Zsum} \psi(t, Z_{t+h}) - \psi(t,Z_t) - \theta(t,Z_t) (Z_{t+h} -Z_t) \leq |Z_{t+h} -Z_t|.
\end{equation}
Consider the uniform partition $P^{\bbT}_{U,n} = \{ t^{U,n}_k ; 0 \leq k \leq 2^n \}$ where $t^{U,n}_k = \epsilon + k 2^{-n} (T - \epsilon)$
and abbreviate $t^{U,n}_k$ to $t^n_k$. Then, applying (\ref{eq:Zsum}) repeatedly,
\begin{eqnarray*}
\sum_{k=1}^{2^n} \left[ \psi(t^{n}_{k-1}, Z_{t^{n}_{k}}) - \psi(t^{n}_k,Z_{t^{n}_k}) \right]
  + \sum_{k=1}^{2^n} \left[ \psi(t^n_{k},Z_{t^n_{k}}) - \psi(t^n_{k-1} ,Z_{t^n_{k-1}}) \right] \! \!
  & \! - \! & \! \! \sum_{k=1}^{2^n} \theta(t^n_{k-1} ,Z_{t^n_{k-1}}) (Z_{t^n_{k}} -Z_{t^n_{k-1}})  \\
  & \leq & \sum_{k=1}^{2^n} |Z_{t^n_{k}} -Z_{t^n_{k-1}}|
\end{eqnarray*}
The second sum on the top line telescopes. By monotonicity, as we take finer and finer
uniform partitions, the term in the
second line converges:
$\sum_{k=1}^{2^n} |Z_{t^n_{k}} -Z_{t^n_{k-1}}| \rightarrow
\sV(\sP^{\cbbT}_U,Z) \leq \sV^{\cbbT}(Z)$.

For the remaining terms, recalling $[t]_n = \max \{ t^k_n \in P^{\cbbT}_n : t^k_n \leq t \}$, and setting
$[t]^n = \min \{ t_k^n \in P^{\cbbT}_n : t_k^n \geq t \}$
\[ \sum_{k=1}^{2^n} \theta(t^n_{k-1} ,Z_{t^n_{k-1}}) (Z_{t^n_{k}} -Z_{t^n_{k-1}})
= \int_\epsilon^T \theta([u]_n,Z_{[u]_n}) dZ_u \rightarrow \int_\epsilon^T \theta(u,Z_u) dZ_u \]
by bounded convergence,
since $\theta$ is continuous and bounded and $Z$ is of finite variation, and
\[ \sum_{k=1}^{2^n} \left( \psi(t^n_{k},Z_{t^n_{k}})
- \psi(t^n_{k-1} ,Z_{t^n_{k}}) \right) = \int_\epsilon^T \dot{\psi}(u,Z_{[u]^n}) du
\rightarrow \int_\epsilon^T \dot{\psi}(u,Z_u) du \]
again by bounded convergence, using the continuity and boundedness of $\dot{\psi}$.
\end{proof}

\begin{cor}\label{cor:dual0}
Suppose $\overline{\mu} \in \cap_{0 \leq t \leq T} E_t$ and $\mu_0 = \delta_{\overline{\mu}}$. If $\dot{\psi}(t,x)$ and $\theta(t,x)$ are continuous and bounded on
$(0,T)\times I$ then if $Z_0(\omega)=\overline{\mu}$
\begin{equation}
\label{eq:dual0}
\sV^{[0,T]}(Z(\omega)) \geq \psi(T, Z_T(\omega)) -  \int_0^T \dot{\psi}(u,Z_u(\omega)) du
- \int_0^T \theta(u,Z_u(\omega)) dZ_u(\omega) .
\end{equation}
\end{cor}

\subsection{Self-similar marginals}
\label{ssec:selfsimilar}
Suppose that in addition to being increasing in convex order, the family $(\mu_t)_{t \in [0,T]}$ is self-similar in the sense that there exists a
centred continuous random variable $Z$ and some positive exponent $\alpha$
such that $\mu_t = \sL(t^\alpha Z)$. Examples include the Gaussian (Example~\ref{eg:bm})
and continuous uniform (Example~\ref{eg:uniform}) specifications. Note that we must have $\mu_0 \sim \delta_{0}$.

Let $\mu_Z$ be the law of $Z$, with potential $U_Z$, and let $I_Z$ with endpoints $\{\ell_Z, r_Z \}$ be the smallest interval such that
the support of $\mu_Z$ is contained in $I_Z$. Let $\rho_Z$ be the associated density and suppose that the product $z \rho_Z(z)$ is continuously differentiable
on $I_Z$ and that $\rho_Z(z) + z \rho'_Z(z)$ is strictly positive on an open interval $E_Z = (\ell_E, r_E)$ with $\ell_E < 0 < r_E$ and negative on
$(\ell_Z, \ell_E) \cup (r_E,r_Z)$. Suppose
further that $\int_{I_Z} (1 + |z|^2) (\rho_Z(z) + |\rho_Z'(z)|) dz < \infty$. Finally, define $\zeta_Z(z) = - \alpha (\rho_Z(z) + z \rho'_Z(z))$ and suppose
$\sup_{\ell_E < z < r_E} \frac{\zeta_Z(z)^-}{\rho_Z(z)} < K_Z$ for some finite constant  $K_Z$.

If $\rho(t,y)$ and $U(t,y)$ are the density and potential of $\mu_t$ then we have $\rho(t,y) = t^{-\alpha} \rho_Z(y t^{-\alpha})$ and
$U(t,y) = t^\alpha U_Z(y t^{-\alpha})$. Further, if we define $P_Z(y) = \frac{1}{2} \alpha \{ U_Z(y) - y U_Z'(y) \}$, then
\[ Q(t,y) : = \frac{1}{2} \dot{U}(t,y) = t^{\alpha-1} P_Z(yt^{-\alpha}) = \frac{1}{2} \alpha t^{\alpha -1} \{ U_Z(y t^{-\alpha}) - y t^{-\alpha} U'_Z(y t^{-\alpha}) \} \]
is positive since $U_Z$ is the potential of a centred random variable, and then, since $P_Z'' = \xi_Z$,
\[ \dot{\rho}(t,y) = Q''(t,y) = t^{-(\alpha+1)} \zeta_Z (y t^{-\alpha}). \]
It follows that we can define $\gamma_t(dy) = \dot{\rho}(t,y)^- dy$ and $\lambda_t(dy) = \dot{\rho}(t,y)^+ dy$, and that
$m(t) = \int \dot{\rho}(t,y)^- dy = \frac{1}{t} \int_{I_Z} \zeta(z)^- dz$ and $n(t) = t^{\alpha - 1} \int_{I_Z} z \zeta_Z(z)^- dz$ are both finite and well defined.

By definition $R_t(y) = \frac{\dot{\rho}(t,y)^-}{\rho(t,y)} = \frac{1}{t} \frac{\zeta_Z(y t^{-\alpha})^-}{\rho_Z(y t^{-\alpha})} < \frac{1}{t} K_Z$.

For $y \in E_Z$ let $A(y) =A_Z(y)$ and $B(y) = B_Z(y)$ with $A(y) < y < B(y)$ attain the supremum in
\[ 
 {\sup}_{a < y < b} \left\{\frac{P_{Z}(y) + (b - y)P'_Z(y) - P_Z(b)   - P_Z(a)}{b - a} \right\} .
\] 
Then, as in Section~\ref{ssec:constructionfixedt} $A : E_Z \mapsto (\ell_Z,\ell_E)$ and $B : E_Z \mapsto (r_E, r_Z)$ are decreasing.

For $y \in \overline{E}_Z$ define
\[ \Theta(y) = \int_{0}^y \frac{2}{B(z)-A(z)} dz; \hspace{10mm} \Psi(y) = \int_{0}^y \frac{2y - B(z) - A(z)}{B(z)-A(z)} dz. \]
Extend these definitions to $y \notin \overline{E}_Z$ by setting
\begin{eqnarray}
\Theta(y) & = & \left\{ \begin{array}{ll} \Theta( A^{-1}(y)); & y < \ell_E \\
                                           \Theta( B^{-1}(y)); \; \; & y > r_E \end{array} \right\}
    \label{eq:Thetadef}
    \\
 \Psi(y) &  = &  \left\{ \begin{array}{ll} \Psi( A^{-1}(y)) + (A^{-1}(y)-y)(1 - \Theta(A^{-1}(y))); & y < \ell_E \\
                                           \Psi( B^{-1}(y)) + (y-B^{-1}(y))(1 + \Theta(B^{-1}(y))); \; \; & y >r_E
                                           \end{array} \right\}
                                           \nonumber
\end{eqnarray}
Finally, define $a_t(y) = t^\alpha A(y t^{-\alpha})$, $b_t(y) = t^\alpha B(yt^{-\alpha})$,
$\theta_t(y) = \Theta(y t^{-\alpha})$ and $\psi_t(y) = t^\alpha \Psi(y t^{-\alpha})$. Then, for $y \in E_t = (t^\alpha \ell_E, t^\alpha r_E)$,
\[ \theta_t(y) = \int_{0}^y \frac{2}{b_t(z)-a_(z)} dz \hspace{10mm} \psi_t(y) = \int_{0}^y \frac{2y - b_t(z) - a_t(z)}{b_t(z)-a_t(z)} dz , \]
and $\psi_t(y) \leq \psi_t(x) + (y-x) \theta_t(x) + |y-x|$.

Then, for any process $Y$ with marginals $(\mu_t)_{t \in [\epsilon,T]}$,
\begin{eqnarray*}
V^{[\epsilon,T]}(Y)   &\geq & \int_\epsilon^T dt \int dy \psi_t(y)  \dot{\rho}(t,y) \\
               & = &\int_\epsilon^T dt \int dy t^\alpha \Psi(y t^{-\alpha}) t^{-(\alpha+1)}\zeta_Z(y t^{-\alpha}) \\
               & = & \int_\epsilon^T dt \; t^{\alpha -1} \int dz \Psi(z) \zeta_Z(z) \\
               & = &   \frac{(T^\alpha - \epsilon^\alpha)}{\alpha}\int dz \Psi(z) \zeta_Z(z)
\end{eqnarray*}
with equality for the process $X$ constructed in Section~\ref{ssec:construction} using the martingale transports defined via (\ref{eq:pidef}).
Letting $\epsilon \downarrow 0$ we obtain $V^{[0,T]}(Y) \geq \frac{T^\alpha}{\alpha}\int dz \Psi(z) \zeta_Z(z)$.
\begin{rem}{\rm
Note that $\dot{\psi}(t,x) = \alpha t^{\alpha-1} [\Psi(xt^{-\alpha}) - x t^{-\alpha} \Psi'(xt^{-\alpha}) ] \leq \alpha t^{\alpha - 1} J$ where
$J = \sup_y |\Psi(y) - y \Psi'(y)|$. Then from the shape of $\Psi$ (concave, then convex, then concave again) we find $J = \max \{ \Psi(\ell_E) + (\Theta(\ell_E)+1) |\ell_E|;
\Psi(r_E) + (1 - \Theta(r_E)) r_E;  \Psi'(\ell_E)\ell_E - \Psi(\ell_E); \Psi'(r_E)r_E - \Psi(r_E)\} < \infty$. Similarly $\sup_x |\theta_t(x)| = \sup_x |\Theta(x)| \leq 2$.

Hence on each interval of the form $[\epsilon, T]$, with $\epsilon > 0$ (\ref{eq:dual}) holds (and if $\epsilon = 0$ then (\ref{eq:dual0}) holds)
and we have a pathwise inequality to
complement the inequality in expectation.
}\end{rem}

\begin{eg}[Brownian motion]{\rm
For Brownian motion there do not exist explicit forms for $\Theta$ or $\Psi$. Nonetheless we find $V^{[0,T]}(Y) \geq C \sqrt{T}$ where $C$ is a finite constant
which can evaluated numerically:
\[ C = 2 \int \Psi(y) \zeta_Z(y) dy  = 2 \int \Psi(y) (- y \rho_Z(y) )' dy .\]
The inequality is tight; in particular, there exists a fake Brownian motion with finite total variation on $[0,T]$.
Moreover $C$ is bounded above by
\[ C \leq 4 \int_1^\infty \Psi(y) (-y \rho_Z(y))' dy
= 4 \int_1^\infty (z^3 - z) e^{-z^2/2}
\frac{dz}{\sqrt{2 \pi}} = \sqrt{\frac{32}{\pi}}e^{-1/2}. \]
Here we use symmetry about zero, the bounds $0 \leq \Psi(y) \leq |y|$, and the fact that $(-y \rho_Z(y))' =
(y^2-1)e^{-y^2/2} / \sqrt{2 \pi}$
is negative on $|y| \leq 1$.
}
\end{eg}

\begin{eg}
\label{eg:uniform2}{\rm
Consider the continuous uniform example (Example~\ref{eg:uniform}) in which $\mu_t \sim U[-t,t]$.
In this case we can write down values for $\Theta$ and $\Psi$ directly.
Note that in (\ref{eq:uniformdual}) below we have exploited the fact that when $x \notin [-1,1]$ the inequality
$L_t(x,y) = 0$ is only tight at $y=x$ to give a simpler form for $\Theta$ than that given in
(\ref{eq:Thetadef}).

Let $\Psi(x) = x^2 \wedge 1$ and $\Theta(x) = x I_{\{ -1<x<1 \} }$. Then, for all $x,y \in \R$,
\begin{equation}
\label{eq:uniformdual}
 \Psi(y) \leq \Psi(x) + \Theta(x)(y-x) + |y-x|,
 \end{equation}
or equivalently, $|y-x| \geq \Psi(y) - \Psi(x) - \Theta(x)(y-x)$. The inequality is easily
proved by considering the different cases: for example, if $-1<x \leq y < 1$,
\[ \Psi(x) + \Theta(x)(y-x) + |y-x| = x^2 + x(y-x) + (y-x) = y^2 + (y-x)(1-y) \geq y^2 = \Psi(y) \]
and if $-1<x < 1 \leq y$,
\[ \Psi(x) + \Theta(x)(y-x) + |y-x| = x^2 + x(y-x) + (y-x) = 1 + (x+1)(y-1) \geq 1 = \Psi(y) .\]

Now define $\psi(t,x) = t \Psi(x/t)$ and $\theta(t,x) = \Theta(x/t)$. It follows that
\begin{eqnarray*} \psi(t,y) = t \Psi(y/t) & \leq & t \left\{ \Psi(x/t) + \Theta(x/t)\left(\frac{y}{t} - \frac{x}{t} \right) +
\left| \frac{y}{t} - \frac{x}{t} \right| \right\} \\
& = & \psi(t,x) + \theta(t,x) (y-x) + |y-x| .
\end{eqnarray*}

Let $Y= (Y_t)_{0 \leq t \leq T}$ be any martingale with uniform marginals $Y_t \sim U[-t,t]$. Then, disregarding for the moment that the regularity
elements of Assumption~\ref{ass:dispersion} are not satisfied, from (\ref{eq:inequality}) we have
\begin{equation}
V^{[0,T]}(Y) \geq  \lim_{\epsilon \downarrow 0} \int_\epsilon^T dt \int \frac{x^2}{t} q_t(dx)
 =  \int_0^T dt \left[ t \frac{1}{2t} + t \frac{1}{2t} - \int_{-t}^{t} \frac{x^2}{t} \frac{dx}{2 t^2} \right]
 =  \int_0^T \left[ 1 - \frac{1}{3} \right] dt = \frac{2T}{3},
\end{equation}
or alternatively, using the second representation in (\ref{eq:0T}),
\[ V^{[0,T]}(Y) \geq \frac{1}{2 T^2} \int_{-T}^T x^2 dx + \int_0^T dt \int_{-T}^{T} \frac{x^2}{T^2} \frac
{dx}{2T} = \frac{2T}{3}. \]

To give a direct proof of this bound we argue as follows: omitting the superscript $U,n$ on $t^{U,n}_k$,
\begin{eqnarray*}
V^{[0,T]}(Y) \geq
V( P^{[0,T]}_{U,n},Y) & \geq & \sum_{k=1}^{2^n} \E \left[ \frac{Y_{t_k}^2}{t_{k-1}}\wedge t_{k-1} - \frac{Y_{t_{k-1}}^2}{t_{k-1}} \wedge t_{k-1} \right] \\
& = & \sum_{k=1}^{2^n} \left[ \int_{-t_{k}}^{t_k} \frac{dy}{2 t_k} \frac{y^2}{t_{k-1}} \wedge t_{k-1}
-  \int_{-t_{k-1}}^{t_{k-1}} \frac{dy}{2 t_{k-1}} \frac{y^2}{t_{k-1}} \right] \\
& = & \sum_{k=1}^{2^n} \left[ \int_{t_{k-1}}^{t_k} \frac{dy}{t_k} t_{k-1} +  \int_{0}^{t_{k-1}} dy \left\{ \frac{1}{t_k} - \frac{1}{t_{k-1}} \right\}
\frac{y^2}{ t_{k-1}} \right]  \\
&= & \sum_{k=1}^{2^n} \left[ \frac{t_{k-1}(t_k - t_{k-1})}{t_k} - \frac{(t_k - t_{k-1})}{t_k t_{k-1}} \frac{t_{k-1}^2}{3} \right] \\
& = & \sum_{k=1}^{2^n} \frac{2}{3}\frac{t_{k-1}}{t_k} (t_k - t_{k-1}) \stackrel{n \uparrow \infty}{\longrightarrow} \frac{2T}{3}.
\end{eqnarray*}
}
\end{eg}

\end{document}